\documentclass{article}
\usepackage{graphicx}

\usepackage[fleqn]{amsmath}

\usepackage{amsthm}
\usepackage{amssymb}
\usepackage{amsbsy}
\usepackage{amsfonts}
\usepackage{mathrsfs}
\usepackage{amsmath}
\usepackage[all]{xy}
\usepackage{amstext}
\usepackage{amscd}
\usepackage[dvips]{epsfig}
\usepackage{psfrag}
\usepackage{enumerate}
\usepackage{flafter}
\allowdisplaybreaks

\theoremstyle{plain}
\newtheorem{thm}{Theorem}[section]
\newtheorem{prop}[thm]{Proposition}
\newtheorem{cor}[thm]{Corollary}
\newtheorem{lem}[thm]{Lemma}
\theoremstyle{definition}
\newtheorem{exa}[thm]{Example}

\newtheorem{rem}[thm]{Remark}
\newtheorem{defn}[thm]{Definition}

\def\det{\mathop{\mathrm{det}}\nolimits}
\def\Im{\mathop{\mathrm{Im}}\nolimits}
\def\Ker{\mathop{\mathrm{Ker}}\nolimits}
\def\Coker{\mathop{\mathrm{Coker}}\nolimits}
\def\Hom{\mathop{\mathrm{Hom}}\nolimits}

\def\F{\mathop{\mathbb{F}}\nolimits}

\newcommand{\tri}{{ \lhd}}

\newcommand{\col}{{\rm Col}}

\newcommand{\lra}{\longrightarrow}
\newcommand{\ra}{\rightarrow}

\newcommand{\Z}{{\Bbb Z}}

\newcommand{\id}{{\mathrm{id}}}

\newcommand{\pc}[2]{\mbox{$\begin{array}{c}
\includegraphics[scale=#2]{#1.eps}
\end{array}$}}
\begin{document}
\large
\begin{center}
{\bf\Large Twisted cohomology pairings of knots I; diagrammatic computation}
\end{center}
\vskip 1.5pc

\begin{center}{\Large Takefumi Nosaka}\end{center}\vskip 1pc\begin{abstract}\baselineskip=12pt \noindent
We provide a diagrammatic computation for the bilinear form,
which is defined as the pairing between
the (relative) cup products with every local coefficients and every integral homology 2-class of every links in the 3-sphere.
As a corollary, we construct bilinear forms on the twisted Alexander modules of links.
\end{abstract}

\begin{center}
\normalsize
\baselineskip=11pt
{\bf Keywords} \\
\ \ \ Cup product, Bilinear form, knot, twisted Alexander polynomial,
group homology, quandle \ \ \
\end{center}

\baselineskip=13pt

\tableofcontents

\large
\baselineskip=16pt

\section{Introduction}
The cup products and pairings of connected compact $C^{\infty}$-manifolds $Y$ have a long history, and possess powerful information in topology.
As is classically known from algebraic surgery theory, if $Y$ is simply connected and closed with $ \mathrm{dim}(Y)\geq 6$, then
the homeotype of $Y$ is almost characterized by cup products and some characteristic classes.
Furthermore, there are also some studies for non-simply connected cases, although
the cases have many obstruction and difficulties, such as the $s$-bordism theorem and $\mathbb{L}$-theory
and Blanchfield duality in high dimensional topology (see \cite{Bla,CS,Mil,Hil}).
Meanwhile, in low dimensional topology, it is important to analyse quantitatively the fundamental group $\pi_1(Y)$ (cf. the geometrization conjecture).
That being said, as in the interaction in \cite{CS,COT},
it is sensible to ask how applicable the study of the cup products in high dimensional topology is to that in low one.

This paper focuses on twisted pairings arising from any group homomorphism $\pi_1(Y) \ra G $,
which are constructed in simple and general situations as follows:
Take a relative homology $n$-class $\mu \in H_{n} ( Y ,\partial Y ;\Z )$, and
a right $G$-module $M$ and a $G$-invariant multilinear function $ \psi : M^n \ra A $ for some ring $A$.
Then, we can easily define the composite map
\begin{equation}\label{kiso}
H^1( Y,\partial Y ; M )^{ \otimes n} \xrightarrow{\ \ \smile \ \ } H^{n}( Y ,\partial Y ; M^{\otimes n } )
\xrightarrow{ \ \ \langle \bullet , \ \mu \rangle \ \ } M^{\otimes n } \xrightarrow{\ \ \langle \psi, \ \bullet \rangle \ \ }A
. \end{equation}
Here $M$ is regarded as the local coefficient of $Y$ via $f$, and
the first map $\smile$ is the cup product, and the second (resp. third) is defined by the pairing with $ \mu$ (resp. $\psi$).

However, in general, the linear form has a critical difficulty that the relative cup product in $H^*( Y, \partial Y ; M ) $ seems speculative and uncomputable from definitions.
Actually, even if $Y$ is a surface with orientation 2-class $\mu $, the bilinear 2-form \eqref{kiso}
is complicated and includes an important example:
Precisely, if $G$ is a semisimple Lie group with Killing form $\psi $ and Lie algebra $\mathfrak{g}=M$,
the 2-form \eqref{kiso} yields a symplectic structure on the flat moduli space $\Hom (\pi_1(Y), G)/\! \!/G$ away from singular points, which is universally summarized as the Goldman Lie algebra \cite{G}.
Furthermore, concerning 3-manifolds $Y$, similar difficulties appear in ``the twisted Alexander modules $H_1(Y ; M)$" \cite{Wada,Lin};
Precisely, whereas the study has provided some topological applications (see \cite{FV,Hil}),
few papers have addressed linear forms on $H_1(Y ; M)$.
In addition,
in analysing some pairings of 3-dimensional links, some boundary conditions occur elaborate difficulty appearing in relative (co)homology; see, e.g., \cite{Bla,BE,COT}, \cite[Chapters 4--8]{Hil}.

\

In the series starting from this paper, 
we address 3-dimensional case
where $Y_L$ is the 3-manifold which is obtained from the 3-sphere by removing an open tubular
neighborhood of a link $L$, i.e., $Y_L =S^3 \setminus \nu L.$
Notice 
the relative homology groups
$$ H_3 ( Y_L ,\partial Y_L;\Z ) \cong \Z, \ \ \ \ \ \ H_2 (Y_L ,\partial Y_L ;\Z ) \cong \Z^{\# \pi_0( \partial Y_L)},$$
which are generated by the fundamental 3-class $[Y_L, \partial Y_L]$ and by some Seifert surfaces in $S^3 \setminus \nu L$, respectively.
We should emphasize that it is not easy to directly describe the 3-class and Seifert surfaces.
This point often appears to be a difficulty in many studies (see \cite{COT,Hil,Mil,T}),
e.g., the $A$-polynomial, Milnor link-invariant, and Chern-Simons invariant \cite{Zic}.

Nevertheless, this paper focuses on the bilinear case with $ n=2$, and we succeed in describing a formulation of computing the twisted pairings
\eqref{kiso} with respect to every representation $f: \pi_1(S^3 \setminus \nu L) \ra G$ of every link group (Theorem \ref{mainthm2}).
Namely, the twisted pairing \eqref{kiso} turns out to be computable from only a link diagram,
with describing no Seifert surfaces.
Actually, we can calculate the bilinear forms with respect to some representations (e.g., see Sections \ref{Lexa1}--\ref{rei.8} for the trefoil and figure eight knots), and observe some interesting phenomena.
Moreover, the subsequent paper \cite{Nos5} will show that the setting \eqref{kiso} recovers
three classical pairings: the Blanchfield pairing, twisted cup products of infinite cyclic covers, and
the Casson-Gordon local signature; hence, the main theorem enables us to compute the classical pairings.
The third paper \cite{Nos6} will deal with trilinear cases of \eqref{kiso}, i.e., $n=3$.
Furthermore, the computation of the $(m, m)$-torus link in Propositions \ref{aa11c}--\ref{aa1133c} will
be used in the studies of 4-dimensional Lefschetz fibrations; see \cite{Nos4} for the detail.
In summary, our viewpoint sheds some concrete light on the relative cup product not normally considered,
with applications including some classical topology.


Finally, we roughly explain the relation of 
relative cohomologies from diagrammatic viewpoints.
The key 
is the diagrammatic link-invariant obtained from ``quandle cocycles" \cite{CJKLS,CKS,IIJO},
where quandle is an algebraic system.
In fact, the formulation of computing the bilinear forms formulates
a generalization of the invariants associated with a certain class of quandles. 
The theorem \ref{mainthm2} implies that the link-invariants exactly coincide with the bilinear maps \eqref{kiso}.
In particular, our result gives a topological interpretation of some quandle cocycle invariants, and stress a topological serviceability of quandle theory. 

Moreover, we emphasize that
this discussion in link cases (under a weak condition of $f$) gives explicitly a homomorphism $\mathcal{L}$ from the homology $ H_1 (Y_L ;M) $ to the cohomology $ H^1 (Y_L , \partial Y_L ;M) $.
As in \cite{FV,Hil,Lin}, the former $ H_1 (Y_L ;M) $ is defined from Fox derivation, and seemingly to be a bilinear form.
However, we show a commutative diagram which relates the Fox derivation to the quandle condition (Lemma \ref{busan}), and obtain the map $\mathcal{L}$.
The condition of $f$ is compatible with linear representations
$ \pi_1(S^3 \setminus L) \ra GL_n(\widetilde{R})$ of link groups, where $\widetilde{R}$ is a Noetherian UFD which factors through the abelianization of $ \pi_1 (Y_L)$;
the associated $ H_1 (Y_L ;M)$ is called the twisted Alexander module, and has some studies \cite{FV,Hil,Wada,Lin}.
In conclusion, by composing \eqref{kiso} with $\mathcal{L}$,
we succeed in introducing bilinear forms on the twisted Alexander modules $ H_1 (Y ;M) $ of a link.


\

This paper is organized as follows. Section 2 formulates the twisted pairing by means of the quandle cocycle invariants,
and states the main theorems. Section 3 describes some computation.
In application, Section 4 introduces bilinear forms on twisted Alexander dual modules.
Section 5 proves the theorems, after reviewing the relative group cohomologies.

\

\noindent
{\bf Notation.} Every link $L$ is smoothly embedded in the 3-sphere $S^3$ with orientation.
We write $Y_L$ for the 3-manifold which is obtained from $S^3$ by removing an open neighborhood of $L$.
Further, we denote by $ \pi_L $ the fundamental group $\pi_1(Y_L)$,
and denote by $\# L$ the number of the link component, i.e., $\# L =|\pi_0( \partial Y_L)|.$
Furthermore, we fix a group homomorphism $f:\pi_L \ra G$,
and by $A$ we mean an abelian group.

\section{Results; diagrammatic formulations of the bilinear forms}\label{ss1}
Our purpose in this section is to state the main results in \S \ref{sss3}.
For this purpose, \S \ref{sss2} starts by reviewing quandles,
and formulates some link-invariants of bilinear forms.

\subsection{Preliminary; formulations of the bilinear maps}\label{sss2}
We will need some knowledge of quandles before proceeding.
A {\it quandle} \cite{Joy} is a set, $X$, with a binary operation $\tri : X \times X \ra X$ such that
\begin{enumerate}[(I)]
\item The identity $a\tri a=a $ holds for any $a \in X. $
\item The map $ (\bullet \tri a ): \ X \ra X$ defined by $x \mapsto x \tri a $ is bijective for any $a \in X$.
\item The identity $(a\tri b)\tri c=(a\tri c)\tri (b\tri c)$ holds for any $a,b,c \in X. $
\end{enumerate}
\noindent
For example, every group $ G$ is made into a quandle with the operation $g \lhd h= h^{-1}gh \in G$.
Moreover, let us explain a broad class of quandles on which this paper focuses.
Take a right $G $-module $M$, that is,
a right module of the group ring $\Z[G ]$.
Let $X= M \times G$, and define a quandle operation on $X$ by
\begin{equation}\label{kihon} \lhd: (M \times G) \times (M \times G)\lra M \times G, \ \ \ \ \ \ (a,g,b,h) \longmapsto (\ (a-b)\cdot h +b, \ h^{-1}gh \ ). \end{equation}
This quandle was first introduced in \cite[Lemma 2.2]{IIJO}.

Next, let us recall $X$-colorings, where $X$ is a quandle.
Let $D$ be an oriented link diagram of a link $L\subset S^3.$
An $X$-{\it coloring} of $D$ is a map $\mathcal{C}: \{ \mbox{arcs of $D$} \} \to X$
such that
$\mathcal{C}(\alpha_{\tau}) \lhd \mathcal{C}(\beta_{\tau}) = \mathcal{C}(\gamma_{\tau})$
at each crossings of $D$ illustrated as
the figure below. 
Let $\mathrm{Col}_X(D) $ denote the set of all $X$-colorings of $D$.
For example, for a group $X=G$ with the conjugacy operation,
the Wirtinger presentation implies that the set $ \mathrm{Col}_X(D) $ is bijective to the set of group homomorphisms
$\pi_L \ra G$. Namely
\begin{equation}\label{kihon22} \mathrm{Col}_G(D) \longleftrightarrow \mathrm{Hom}_{\rm gr}(\pi_L , G ). \end{equation}
\vskip -0.17pc

\begin{figure}[htpb]
\begin{center}
\begin{picture}(100,26)
\put(-22,27){\large $\alpha_{\tau} $}
\put(13,25){\large $\beta_{\tau} $}
\put(13,-12){\large $\gamma_{\tau} $}

\put(-36,3){\pc{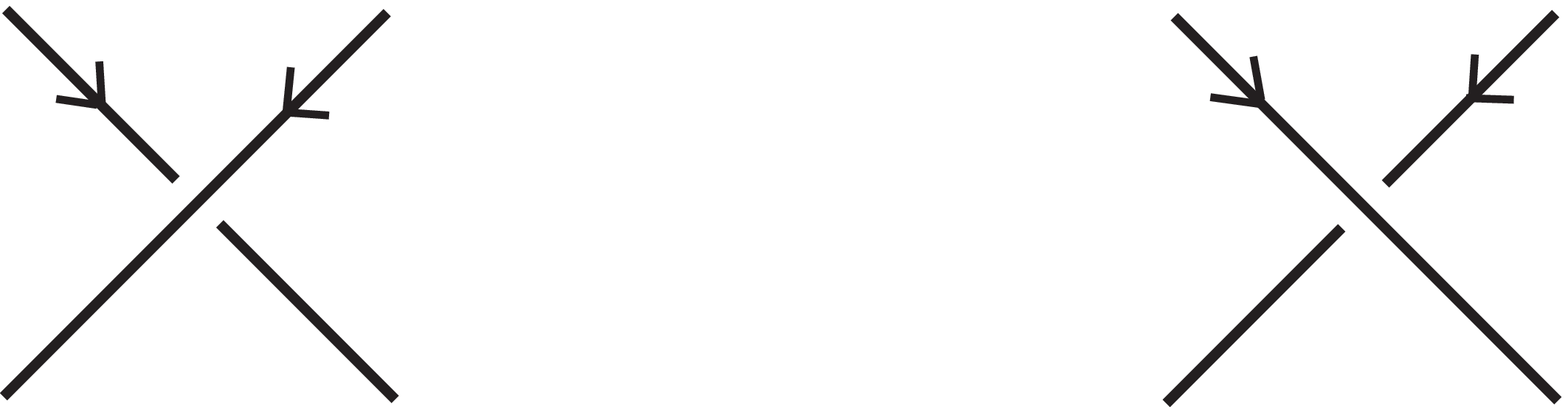}{0.27304}}

\put(143,-13){\large $\beta_{\tau} $}
\put(103,-13){\large $\alpha_{\tau} $}
\put(143,25){\large $\gamma_{\tau} $}
\end{picture}
\end{center}
\end{figure}


We will explain the subset \eqref{kihon2294} and a decomposition \eqref{skew24592} below.
By assumption, via \eqref{kihon22}, we can regard the homomorphism $f$ as a $G$-coloring of a link-diagram $D$.
Take the quandle $X= M \times G$ in \eqref{kihon} and the projection $p_G: X \ra G$.
Then, we define the set of lifts of $f$ as follows:
\begin{equation}\label{kihon2294} \mathrm{Col}_X(D_{f}):= \{ \ \mathcal{C} \in \mathrm{Col}_X(D) \ | \ p_G \circ \mathcal{C} =f \ \}. \end{equation}
It is worth noticing that the set $\mathrm{Col}_X(D) $ is regarded as a subset of the product $X^{\# \{ \textrm{arcs of }D \}}$.
Hence, the subset $\mathrm{Col}_X(D_{f})$ is made into an abelian subgroup of $M^{\# \{ \textrm{arcs of }D \}}$ according to
the linear operation \eqref{kihon}.
Further, we can easily see that the diagonal subset $M_{\rm diag} \subset M^{\# \{ \textrm{arcs of }D \}}$
is a subset of $\mathrm{Col}_X(D_{f}) $ as a direct summand in $\mathrm{Col}_X(D_{f}) $.
Denoting another summand by $ \mathrm{Col}^{\rm red}_X(D_{f}) $, we have a direct decomposition
\begin{equation}\label{skew24592}\mathrm{Col}_X(D_{f}) \cong \mathrm{Col}^{\rm red}_X(D_{f}) \oplus M .\end{equation}

Furthermore, one introduces a bilinear form on the $\Z$-module $ \mathrm{Col}_X(D_{f})$ as follows (Definition \ref{deals3}).
Taking another $G$-module $M'$, let $\psi : M \times M' \ra A$ be a bilinear map over $\Z$.
Moreover, we assume that this $\psi$ is $G$-invariant. Namely,
$$ \psi(x \cdot g, y \cdot g)= \psi(x , y ) \ \ \ \ \ \ \ \mathrm{for \ any} \ x \in M, \ y\in M' \mathrm{ \ and \ } g \in G.$$
Considering the associated quandle $X'= M' \times G$, define the map 
$ \varphi_{\psi} : X \times X' \ra A$ by setting
\begin{equation}\label{skew245}
\varphi_{\psi} \bigl( (y_1,g_1),\ (y_2,g_2) \bigr) = \psi \bigl( y_1 ,\ y_2 \cdot(1-g_2^{-1} ) \bigr), \end{equation}
which is first introduced \cite[Corollary 4.7]{Nos2}. Furthermore,
recall from (\ref{kihon2294}) the set $ \mathrm{Col}_{Z}^{\rm red}(D_{f})$ associated with $Z=X$ or $Z=X'$.
Inspired this, we define
\begin{defn}\label{deals3}
Let $X$ and $X'$ be as above, let $D = K_1 \cup \cdots \cup K_{\# L}$ be a link diagram, where $K_1, \dots, K_{\# L}$ are connected components.
For $1 \leq \ell \leq \# L$, we define a map $\mathcal{Q}_{\psi,\ell }$ by
$$ \mathrm{Col}_{X}^{\rm red}(D_{f}) \times \mathrm{Col}_{X'}^{\rm red} ( D_{f}) \lra A; \ \ \ ( \mathcal{C}, \mathcal{C}' ) \longmapsto \sum_{\tau} \epsilon_{\tau} \psi \bigl( x_{\tau}-y_{\tau} , \ y_{\tau} ' \cdot (1 - h_{\tau}^{-1}) \bigr), $$
where $\tau $ runs over all the crossings such that the under-arc is from the component $K_\ell$, and $\epsilon_{\tau} \in \{ \pm 1\}$ is the sign of $\tau$ according to the figure below.
Furthermore, the symbols $(x_{\tau}^{\bullet}, y_{\tau})\in X $ and $ (y_{\tau}^{\bullet} , \ h_{\tau})\in X' $ are the colors around the crossing $\tau.$
\end{defn}
\begin{figure}[htpb]
\begin{center}
\begin{picture}(120,26)
\put(-72,24){\large $(x_{\tau} , g _{\tau} ) $}
\put(13,24){\large $(y_{\tau} , h _{\tau} ) $}
\put(82,24){\large $(x_{\tau}' , g _{\tau} ) $}
\put(168,24){\large $(y_{\tau}' , h _{\tau} ) $}

\put(-83,2){\Large $ \mathcal{C} $}
\put(73,2){\Large $ \mathcal{C}' $}
\put(-33,5){\large $ \tau $}
\put(123,5){\large $ \tau $}

\end{picture}
\end{center}
\end{figure}

\subsection{Statements of the main theorems}\label{sss3}
As mentioned in the introduction,
we will show (Theorem \ref{mainthm2}) that the twisted cohomology pairing \eqref{kiso} is described as the previous bilinear maps 
and states some corollaries. 
The proofs of the theorems appear in \S \ref{Lbo33222}.

Next, from the view of this theorem we will reformulate the bilinear form $ \mathcal{Q}_{\psi} $ defined in \S \ref{sss2}. As mentioned in the introduction, recall
the isomorphism
$ H_{2}(Y_L ,\partial Y_L;\Z ) \cong \Z^{\# L}$
with a basis $ \mu_1, \dots, \mu_{ \# L}$ which correspond to the longitudes (or Seifert surfaces) in $S^3 \setminus L.$

\begin{thm}\label{mainthm2}Let $Y_L$ be a link complement in $S^3$ as in \S 1.
Regard the $G$-modules $M$ and $M'$as a local systems of $Y_L$ via 
$f: \pi_1(Y_L) \ra G$. 

Then, there are isomorphisms
\begin{equation}\label{g21gg33} \mathrm{Col}_{X } (D_{f}) \cong H^1(Y_L , \ \partial Y_L ;M ) \oplus M, \ \ \ \ \ \mathrm{Col}_{X }^{\rm red} (D_{f}) \cong H^1(Y_L , \ \partial Y_L ;M ) . \end{equation}
such that 
the bilinear form $ \mathcal{Q}_{\psi, \ell}$ on $ \mathrm{Col}_{X^{(')}}^{\rm red} (D_{f})$ is
equal to the following composite (cf. \eqref{kiso}):
$$ H^1( Y_L ,\partial Y_L ; M )\otimes H^1( Y_L ,\partial Y_L ; M' ) \xrightarrow{\ \ \smile \ \ } H^2( Y_L ,\partial Y_L ; M\otimes M' )
\xrightarrow{ \ \langle \bullet , \mu_\ell \rangle \ } M \otimes M' \xrightarrow{ \ \ \langle \psi, \bullet \rangle \ \ }A . $$
\end{thm}
As a concluding remark,
we should emphasize again that
we can compute the cohomology pairing of links from only a link diagram without describing longitudes (or Seifert surfaces) in $S^3 \setminus L$.
Moreover, as seen in Definition \ref{}, the pairing seems computable in an easy way (see \S \ref{Lb162r3} for the examples). 

In addition, we see that 
the bilinear form in Definition \ref{deals3}
formulates a generalization of the quandle cocycle invariants \cite{CJKLS,IIJO} with respect to quandles of the forms $X= X'= M \times G $.
The link invariants \cite{CJKLS,CKS}, constructed from
a quandle $X$ and a map $\Psi: X^2 \ra A $
which satisfies ``the quandle cocycle condition",
were defined to be a certain map $\mathcal{I}_{\Psi}: \mathrm{Col}_X(D )\ra A $.
%
Then, we can see that the map $ \varphi_{\psi} $ in \eqref{skew245} which is a quandle 2-cocycle,
and verify the equality
$\mathcal{I}_{\varphi_{\psi} } = \mathcal{Q}_{\psi } \circ \bigtriangleup $ by construction,
where $\bigtriangleup : \mathrm{Col}_X(D )\ra \mathrm{Col}_X(D )^2$ is
the diagonal map.
To sum up, as a result of Theorem \ref{mainthm2},
we have succeeded in describing entirely a topological meaning of the quandle cocycle invariants.

\

Next, we mention two properties which are used in the papers \cite{Nos5,Nos4}.
We then discuss a non-degeneracy or duality of $ \mathcal{Q}_{\psi, \ell} $.
However, 
we should mention the connecting map $\delta^*: H^{0} ( \partial Y_L ; M ) \ra H^1( Y_L ,\partial Y_L; M)$.
Actually, if $M=M'$ and if $\mathbf{x} \in \mathrm{Im}(\delta^*) $, then
the two vanishings
$ \mathcal{Q}_{\psi, \ell}( \mathbf{x},\mathbf{y})= \mathcal{Q}_{\psi, \ell}( \mathbf{y},\mathbf{x})=0$
hold for any $ \mathbf{y} \in \mathrm{Col}_{X } (D_{f}) $.
(cf. Theorem \ref{mainthm14} later).
\begin{cor}[See \S \ref{yy43} for the proof.]\label{mainthm1}Let $Y_L$ be a link complement in $S^3$ as in \S 1.
For each link component $\ell$, fix a meridian $\mathfrak{m}_{\ell} \in \pi_1(Y_L) $.
If the maps $ \mathrm{id}_M-f( \mathfrak{m }_{\ell } ): M \ra M $ are isomorphisms for any $\ell \leq \#L$, 
then the inclusion $(Y_L, \emptyset) \ra (Y_L, \partial Y_L )$ induces
the isomorphisms
$ H^1(Y_L ,\partial Y_L ;M) \cong H^1(Y_L ;M)$ and $ \mathrm{Im}(\delta^*) \cong 0$.

In particular, 
the decomposition in \eqref{g21gg33} is written as 
$\mathrm{Col}_{X }^{\rm red} (D_f ) \cong H^1(Y_L ;M ) .$
\end{cor}


On the other hand, the invariance with respect to conjugacy is immediately shown;
\begin{cor}\label{lem373d}
Let $\phi$ be a $G$-bilinear map as above,
and let $f$ and $f' $ be two homomorphisms $\pi_L \ra G$.
If there is $g \in G$ such that
$f(\mathfrak{m})= g^{-1}f'(\mathfrak{m})g \in G$ for any meridian $\mathfrak{m} \in \pi_L$,
then the resulting bilinear maps $\mathcal{Q}_{\psi ,\ell }$ and $\mathcal{Q}'_{\psi ,\ell } $ are equivalent.
\end{cor}

Finally, we give a special corollary of Theorem \ref{mainthm2},
when $G $ is the free abelian group $\Z^{\# L}$ and $f: \pi_L \ra \Z^{\# L}$ is the canonical abelianization.
Writing $t_1, \dots, t_{\# L }$ for generators of $ \Z^{\# L} $,
we can consider the $G$-module $M$ to be a module over
the Laurent polynomial ring $ \Z[t_1^{\pm 1}, \dots, t_{\# L }^{\pm 1}] $.
Then, Theorem \ref{mainthm2} immediately deduces a topological meaning on the set of colorings.
\begin{cor}\label{le359}
Let $L$ be a link, and $f$ be its abelianization $\mathrm{Ab} : \pi_L \ra G= \Z^{\# L} $. Take a $ \Z[t_1^{\pm 1}, \dots, t_{\# L }^{\pm 1}] $-module $M$.
Then, we have a $ \Z[t_1^{\pm 1}, \dots, t_{\# L }^{\pm 1}] $-module isomorphism
$$ \mathrm{Col}_X(D_{f}) \cong H^1( Y_L, \partial Y_L ;M ) \oplus M. $$
\end{cor}
\begin{rem}
Let us compare Theorem \ref{mainthm2} with the previous papers, and mention some leap forwards.
Concerning the set $\mathrm{Col}_X(D_{f})$, many papers have dealt with only the case $G=\Z$ (which is commonly called ``the Alexander quandle $X$"; see \cite{CJKLS}).
However, as seen in \cite[\S 12]{Joy} or \cite{CDP} and references therein,
while some papers discussed a connection to Alexander polynomials in knot case,
few papers analysed a relation between ${\rm Col}_X(D)$ and Alexander polynomials (or module) if $\# L >1$.
Corollary \ref{le359} implies a conclusive remark that the set $ {\rm Col}_X(D)$ is interpreted not only by usual (group) homologies,
but by relative ones.


\end{rem}

\section{Bilinear forms on the twisted Alexander modules of links}\label{LFsu32b12r2223}
The purpose of this section is to define bilinear forms
on the twisted Alexander (dual) modules (Definition \ref{twisted}).
According to most papers on the twisted polynomial (see \cite{FV,Wada,Lin}), we mean by $R$ a (commutative)
a Noetherian unique factorization domain (henceforth UFD), with involution $\bar{}: R \ra R$.

\subsection{Preliminaries}\label{ss131}
For this purpose, we start by briefly reviewing
the twisted Alexander module
associated with two group homomorphisms
$$f_{\rm pre} : \pi_L \ra GL_n(R), \ \ \ \mathrm{and } \ \ \ \rho :\pi_L \ra \Z^m$$
for some $m \in \mathbb{N}$.
Identifying the group ring, $R [\Z^{m}]$ , of $\Z^{m} $ with the polynomial ring $R [t_1^{\pm 1}, \dots, t_{m}^{\pm 1}]$,
the map $\rho$ is extended to a representation $ \pi_L \ra \mathrm{End}_R( R [t_1^{\pm 1}, \dots, t_{m}^{\pm 1}]) $.
Hence, tensoring this $ \rho$ with $ f_{\rm pre}$, we have a representation
$$ \rho\otimes f_{\rm pre} : \pi_L \lra GL_n ( R [t_1^{\pm 1}, \dots, t_{m }^{\pm 1}]) .$$
Thus, the associated first homology $H_1(Y_L ;R[t_1^{\pm 1}, \dots, t_{m }^{\pm 1}]^n ) $ is commonly called {\it the twisted Alexander module} associated with $f_{\rm pre} $; see a survey \cite{FV} on twisted Alexander polynomials.

This Alexander module can be described from %
the Fox derivative as follows.
Take a diagram $D$ with $\alpha_D = \beta_D$, where
$\alpha_D $ (resp. $\beta_D$) is the number of the arcs (resp. crossings).
Let us denote this $\alpha_D$ by $\alpha$ in short, and consider the Wirtinger presentation 
$\langle x_1 , \ldots , x_{\alpha} | r_1 , \ldots , r_{\alpha} \rangle $ of $\pi_L$.
Let $F_m$ be the free group of rank $m $.
Here, recall that 
there uniquely exists, for each $x_j $, a Fox derivative
$ \frac{\partial \ \ }{\partial x_j} : F_\alpha \ra R [\Z^{m}][F_\alpha ]$
with the following two properties:
$$ \frac{\partial x_i}{\partial x_j} = \delta_{i,j}, \ \ \ \ \ \ \frac{\partial (uv)}{\partial x_j} = \frac{\partial u }{\partial x_j}v +\frac{\partial v}{\partial x_j},$$
for all $u, v \in F_\alpha$.
Then, as is known (see, e.g., Exercise \cite[\S II.5]{Bro}), we
can describe a partial resolution of $\pi_L$ over $A_{(\partial f )}$ as 
\begin{equation}\label{skddew22} (R [\Z^{m}][\pi_L ])^{\alpha }\xrightarrow{ \ \partial_2 \ } (R [\Z^{m}] [\pi_L])^{\alpha } \xrightarrow{ \ \partial_1 \ } R [\Z^{m}] [\pi_L] \stackrel{\epsilon}{\lra} R [\Z^{m}] \lra 0 \ \ \ \ \ \ \ ( \mathrm{exact})
\end{equation}
such that the matrix of $ \partial_2$ is the $(\alpha \times \alpha)$-Jacobian matrix $ ( [\frac{\partial r_i}{\partial x_j}] )$,
and the latter $ \partial_1 $ is defined by $ \partial_1 (\gamma)=1 - \gamma $.
Accordingly, after tensoring with a $R [\Z^{m}]$-module $M$, 
the common quotient $\Ker (\mathrm{id}_M \otimes \partial_1)/\Im (\mathrm{id}_M \otimes \partial_2 )$ is isomorphic to
the first group homology $H_1 (Y_L ;M)$ with local coefficients.

Next, we will set up a localized ring \eqref{aaa} below, and review the twisted Alexander polynomial \cite{Wada,Lin}.
%
For this purpose, assume the non-vanishings
\begin{equation}\label{aaaaa} \mathrm{det}(\mathrm{id} - \rho\otimes f_{\rm pre} (\mathfrak{m} )) \neq 0 \in R[\Z^m]\end{equation}
for every meridian $\mathfrak{m} \in \pi_L$:
A typical example is the case $\rho( \mathfrak{m}) \neq 0 $ in $\Z^m$ for every meridian $\mathfrak{m} $,
such as the abelianization $\pi_L \ra \Z^{\# L}$.
Then, the assumption enables us to define the ring $A_{(\partial f)}$
obtained by inverting the determinants.
Precisely, we set
\begin{equation}\label{aaa}A_{(\partial f)} :=R [t_1^{\pm 1}, \dots, t_{m}^{\pm 1}, \ \prod_{ \ell \leq \# L} \mathrm{det}(\mathrm{id}- \rho\otimes f_{\rm pre} (\mathfrak{m}_{\ell} ))^{-1}
(\mathrm{id}- \overline{\rho\otimes f_{\rm pre} (\mathfrak{m}_{\ell} )} )^{-1} ] .
\end{equation}
We remark that $A_{(\partial f)}$ is also a Noetherian UFD, and that $A_{(\partial f)}$ has the involution $ \bar{ \ } :A_{(\partial f)} \ra A_{(\partial f)}$ defined by $\bar{ t_i} =t_i^{-1}$.
This localization \eqref{aaa} can be interpreted
as a generalization of ``localized Blanchfield pairing" (see \cite[\S 2.6]{Hil}), let us set up 
Then, {\it the twisted Alexander polynomial}, $\Delta_{f}$, is defined to be the $n^2(\alpha-1)^2$ Jacobian of the Fox derivations \eqref{skddew22} subject to \eqref{aaaaa}:
$$\Delta_{f}:= \mathrm{det} \Bigl(([\frac{\partial r_i}{\partial x_j}] )\otimes \mathrm{id}_{A_{(\partial f)}^n })_{1 \leq i,j \leq \alpha -1}\Bigr) /
\mathrm{det}(\mathrm{id} - f_{\rm \mathcal{I}} (x_{\alpha} )) \in A_{(\partial f)}. $$
It is shown \cite{Wada} that the value is independent, up to units, of the choice of the arcs $\alpha$.


In addition, we mention a close relation to the colorings set.
Recall that the subset $\mathrm{Col}_X(D_{f})$ 
is a submodule of the product $M^{\alpha_D}$ according to
the linear operation \eqref{kihon}.
More precisely,
$\mathrm{Col}_X(D_{f})$
can be regarded as the kernel of
the homomorphism
\begin{equation}\label{fukuoka}
\Gamma_{X,D} : M^{\alpha_D} \lra M^{\beta_D} \end{equation}
obtained from \eqref{kihon}. Furthermore, let us examine the cokernel $\Coker (\Gamma_{ X,D }) $:
\begin{lem}\label{busan}
For any link $L$, choose a diagram $D$ with $\alpha_D =\beta_D $.
Consider the quandle $ \overline{X} $ of the form $ M \times G$, where $M$ is the free module $ (A_{(\partial f)})^n $ and $G$ is $GL_n(A_{(\partial f)} ) $.

Then, the cokernel has the following isomorphism
$$\Coker (\Gamma_{ \overline{X},D }) \cong H_1(Y_L ; (A_{(\partial f)} )^n) \oplus( A_{(\partial f)}) ^n.$$
Here the second summand $(A_{(\partial f)})^n$ corresponds to the diagonal subset $A_{\rm diag}$ of $ (A_{(\partial f)})^{n \beta_D } $.
\end{lem}
\begin{proof}
From the definition of the ring $A_{(\partial f)} $ in \eqref{aaa}, every $\mathrm{id}- \rho\otimes f_{\rm pre} (\gamma_i) $ is invertible in $M$;
The map $\mathrm{id}_M \otimes \partial_1$ is a
(diagonally) splitting surjection,
which admits consequently a decomposition
$$\mathrm{Coker} ( \mathrm{id}_M \otimes \partial_2 : M^{\alpha_D }\lra M^{\beta_D } ) \cong H_1 (\pi_L ;M) \oplus M .$$
Here, regarding a crossing $\tau$ illustrated as in Figure \ref{koutenpn},
let us
set up the bijection $\kappa_{\tau } : M \ra M$ which takes $ m$ to $ m -m \cdot \rho\otimes f_{\rm pre} (\alpha_{\tau })$,
and $\kappa_{\tau }' : M\ra M$ which sends $ m$ to $ m -m \cdot \rho\otimes f_{\rm pre}(\gamma_{\tau }) $.
Then, by the direct products with respect to crossings $\tau $, we have the diagram
$${\normalsize
\xymatrix{
0 \ar[r] & \mathrm{Col}_{\overline{X}}(D_f) \ar[r] &M^{\alpha_D } \ar[rr]^{\Gamma_{X,D }}\ar[d]_{\prod_{\tau }\kappa_{\tau } } & & M^{\alpha_D } \ar[r] \ar[d]^{\prod_{\tau } \kappa_{\tau }'}& \mathrm{Coker}(\Gamma_{X,D} ) \ar[r] & 0 & (\mathrm{exact}) \\
& & M^{\alpha_D } \ar[rr]^{ \mathrm{id}_M \otimes \partial_2 } & & M^{\alpha_D } \ar[r] & H_1 (Y_L ;M) \oplus M \ar[r] & 0 & (\mathrm{exact}).
}}
$$
Examining carefully the definitions of $\kappa_{\tau }^{(')}$, $ \partial_2 $, and $ \Gamma_{\overline{X},D} $,
the diagram is commutative. Hence, the vertical maps give the desired decomposition $ \mathrm{Coker}(\Gamma_{\overline{X},D} ) \cong H_1 (Y_L ;M) \oplus M $.
\end{proof}

Finally,
we briefly set up an extension of a bilinear form.
For this, suppose 
a bilinear function $ \psi_{\rm pre} : R^n \times R^n \ra R$ satisfying the $f_{\rm pre}$-invariance 
$$ \psi_{\rm pre} (x,y) = \psi_{\rm pre} ( x \cdot f_{\rm pre}( \mathfrak{m}),\ y \cdot f_{\rm pre}( \mathfrak{m}) ) $$
for any $x,y \in R^n$, and any meridian $ \mathfrak{m} \in \pi_L $.
For an ideal $ \mathcal{I} \subset A_{(\partial f)}$, we let $ \overline{\mathcal{I}} \subset A_{(\partial f)}$ be the ideal consisting of $x \in A_{(\partial f)} $ with $\bar{x} \in \mathcal{I}.$
Then, we can define the bilinear function
\[\psi : ( R^n \otimes_R A_{(\partial f)}/ \overline{\mathcal{I}} ) \times (R^n \otimes_R A_{(\partial f)}/ \mathcal{I}) \lra A_{(\partial f)}/ \mathcal{I}\]
by setting
\begin{equation}\label{aaaaaa} \psi ( x \otimes a_1 , y \otimes a_2) = \psi_{\rm pre} (x,y) \otimes \overline{a_1} a_2,\end{equation}
for $x,y \in R^n $ and $a_1, \ a_2 \in A_{(\partial f)} $. This $\psi $ is $\pi_L $-invariant and sesquilinear over $R[\Z^{ m }] $.


\subsection{Definition}\label{ss33131}

Inspired by Lemma \ref{busan}, we will introduce map from the twisted Alexander module $ H_1 (Y_L ;M)$ to a certain relative cohomology.
After that, by composing with the bilinear form $\mathcal{Q}_{\psi}$, we define a bilinear form on the module $ H_1 (Y_L ;M)$.

For this, consider 
the principal ideal $\mathcal{I}$ generated by $\Delta_f$.
\begin{equation}\label{aaabb} f_{\mathcal{I}}: \pi_L \lra GL_n(A_{(\partial f)} / \mathcal{I}) ,\end{equation}
by passage to $ \mathcal{I}.$
Then, it is sensible to set up the local coefficients $M=(A_{(\partial f)})^n $, $M_{\Delta}:= (A_{(\partial f)}/ \mathcal{I} )^n $ and $M_{\overline{\Delta}}:= (A_{(\partial f)}/ \overline{\mathcal{I}} )^n $ acted on by \eqref{aaabb}.
Here, the reason why we here need 
the ideal $ \mathcal{I}$ is as follows.
In many cases, 
the twisted Alexander modules are often torsion $A_{(\partial f)} $-modules, which are annihilated by $ \Delta_f$ ; see, e.g., \cite{FV,Wada}.
Therefore, to get non-trivial linear function from such modules,
the coefficient ring shall be the quotient $A_{(\partial f)}/( \Delta_f)$.


Next, we will explain Definition \ref{twisted} after introducing two homomorphisms $\mathrm{Adj }$ and $ \mathcal{L}$.
Considering the decomposition $(A_{(\partial f)}) ^{n \alpha_D }=(A_{(\partial f)}) ^{n ( \alpha_D -1) } \oplus M_{\rm diag}$,
we take the restriction $$ \mathrm{res} (\Gamma_{ \overline{X} ,D}) : (A_{(\partial f)}) ^{n ( \alpha_D -1)} \ra ( A_{(\partial f)}) ^{n ( \alpha_D -1)}. $$
of \eqref{fukuoka}.
Then, it follows from Theorem \ref{mainthm2} and Lemma \ref{busan} above that the adjugate matrix
of $\mathrm{res} (\Gamma_{ \overline{X} ,D} ) $ subject to $(\Delta_f) $ yields a well-defined homomorphism
\begin{equation}\label{ppo}\mathrm{Adj } : H_1(Y_L ; ( A_{(\partial f)}) ^n) \lra \mathrm{Col}_X^{\rm red}(D_f) \cong H^1(Y_L, \partial Y_L ; M_{\Delta}).
\end{equation}
where $X$ is $ (M_{\Delta})^m\times G$ as the quotient of $\overline{X}.$
Furthermore, notice that the localization $ R[ \Z^{m } ] \hookrightarrow A_{(\partial f)}$ gives rise to 
the homomorphism
\[ \mathcal{L} : H_1(Y_L ; R[ \Z^{m} ] ^n) \lra
H_1(Y_L ; (A_{(\partial f)})^n).\] 

\begin{defn}\label{twisted}
Let $R$ be a Noetherian UFD, and $A_{(\partial f)}$ and $ G $ be as above.
Take $\mathcal{I}= (\Delta_f) $, and $M_{\Delta} = ( A_{(\partial f)} /\mathcal{I})^n .$
Let $\psi :M_{\overline{\Delta}} \times M_{\Delta} \ra A_{(\partial f)}/\mathcal{I} $ be the bilinear form obtained from $\psi_{\rm pre}$, as in \eqref{aaaaaa}.

Then, we define
the bilinear map from the twisted Alexander module
associated with $ (f_{\rm pre}$, $\psi_{\rm pre})$ to be the following
composite
\[ H_1(Y_L ; R[ \Z^{m } ]^n)^{ \otimes 2} \xrightarrow{ \ \mathrm{Adj }^{\otimes 2}\circ \mathcal{L} ^{\otimes 2} \ }
H^1(Y_L, \partial Y_L ; M_{\overline{\Delta}}) \otimes H^1(Y_L , \partial Y_L; M_{\Delta}) \xrightarrow{ \ \mathcal{Q}_{\psi } \ } A_{(\partial f)}/\mathcal{I} .\]
\end{defn}
%

By definition and Theorem \ref{mainthm2}, we should emphasize that it is not hard to compute the pairing from $ \mathcal{Q}_{\psi } $.
For instance as in Example \ref{extra2}, the computation of $ \mathcal{Q}_{\psi } $ implies the twisted pairing equal to $ 2 x \bar{x}'$.

\

Finally, we end this section by mentioning a duality.
In general, such a duality always do not holds; we should consider a restricted situation:
Let $m=1$, let $R$ be a field $\mathbb{F} $ of characteristic 0, and $ \mathcal{I}$ be the principal ideal $(\Delta_f )$.
Then, we can easily show the following lemma in linear algebra.
\begin{lem}\label{mainthm1224}
Assume that $\Delta_f$ is non-zero and that $\mathrm{det}( t^{-1} \cdot \mathrm{id}_{\F^n} - f_{\rm pre} (\mathfrak{m} )) \neq 0$ for a meridian $\mathfrak{m}\in \pi_L $ is relatively prime to the polynomial
$\Delta_{f}$.
Then the adjugate matrix $\mathrm{Adj }$ in \eqref{ppo} is an $\F[t^{\pm 1}]$-isomorphism.
\end{lem}
As seen in Examples \ref{extra22} and \ref{extra2}, this pairing is often degenerate in many cases,
and is possible to be even zero,
while the classical Blanchfield pairing is non-singular.
However, the subsequent paper \cite{Nos5} will show a duality theorem on the twisted pairings, under assumptions:
\begin{thm}[{A corollary of \cite[Theorem 2.4]{Nos5}}]\label{mainthm14}
Let $m=1$ and let $R$ be a field of characteristic 0, and $ \mathcal{I}$ be the principal ideal $(\Delta_f )$.
Further, assume that $\Delta_f \neq 0$, and $\psi_{\rm pre}$ is nondegenerate. 
If $\mathrm{det}( \mathrm{id}_{\F^n} - t \cdot f_{\rm pre} (\mathfrak{m} )) \neq 0$ for a meridian $\mathfrak{m} \in \pi_L $ is relatively prime to
$\Delta_{f}$ in $ \mathbb{F}[t] $, then the twisted pairing in Definition \ref{twisted} is non-degenerate.
\end{thm}
Here, recall from the known fact of Milnor \cite{Mil2} that all the (skew-)hermitian nondegenerate bilinear forms with isometries $t$ is completely characterised.
In conclusion, if $ \psi$ is (skew-)hermitian, we can obtain computable information from the twisted pairing.


\section{Examples as diagrammatic computations}\label{Lb162r3}
As a result of Theorem \ref{mainthm2} on the twisted pairings,
we will compute the bilinear forms $ \mathcal{Q}_{\psi}$
associated with some homomorphisms $f : \pi_L \ra G$, where $L$ is one among the trefoil knot, the figure eight knot and the $(m,m)$-torus link $T_{m,m}$. The reader may skip this section.

\subsection{The trefoil knot }\label{Lexa1}

\begin{figure}[htpb]
\begin{center}
\begin{picture}(10,60)
\put(-163,23){\large $\alpha_1 $}
\put(-148,45){\large $\alpha_2 $}
\put(-91,23){\large $\alpha_3 $}
\put(-23,51){\large $\alpha_2 $}
\put(-44,19){\large $\alpha_1 $}
\put(15,10){\large $\alpha_4 $}
\put(37,19){\large $\alpha_3 $}

\put(88,51){\large $\alpha_1 $}
\put(87,-2){\large $\alpha_m $}
\put(88,24){\large $\alpha_i $}

\put(91,10){\normalsize $\vdots $}
\put(91,38){\normalsize $\vdots $}

\put(-161,22){\pc{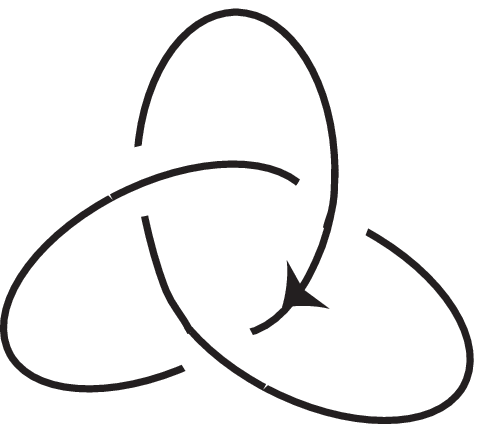}{0.53104}}
\put(-36,22){\pc{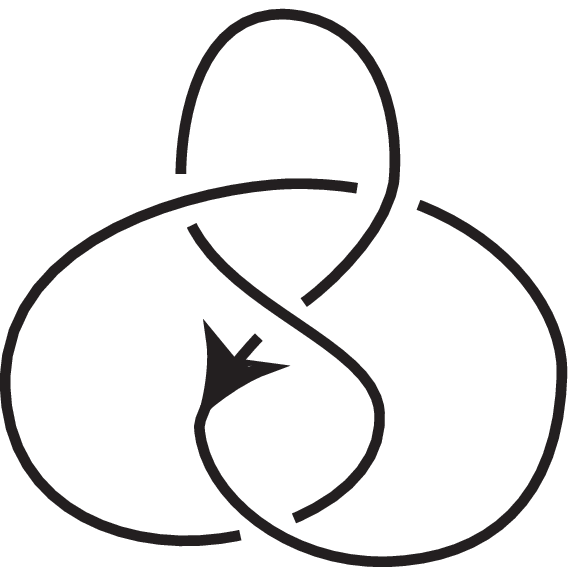}{0.402104}}
\put(96,26){\pc{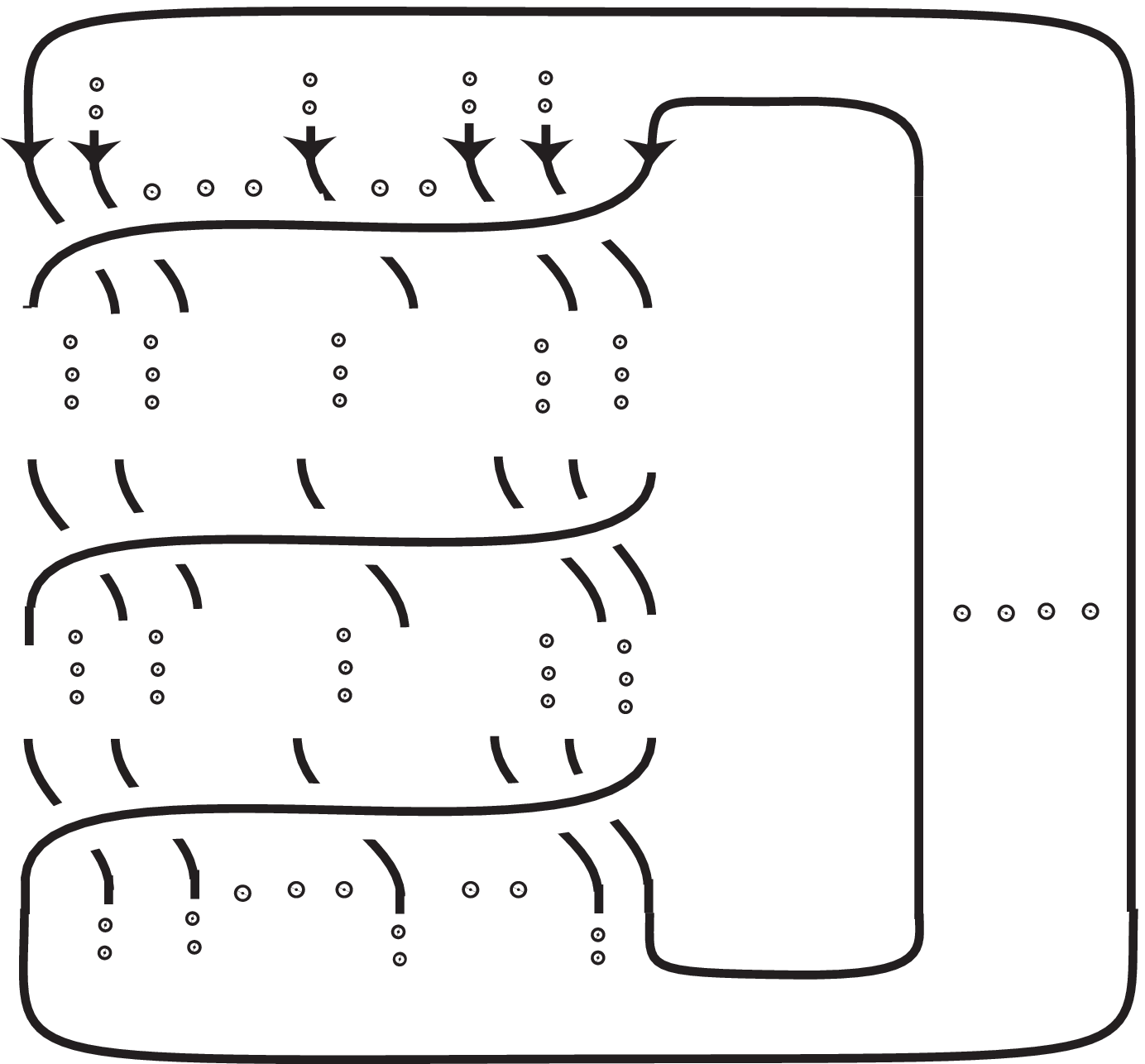}{0.2473104}}
\end{picture}
\end{center}\caption{\label{ftf} The trefoil knot, the figure eight knot and the $T_{m,m}$-torus link with labeled arcs.}
\end{figure}

As a simple example,
we will focus on the the trefoil knot $K $.
Let $D$ be the diagram of $K$ as illustrated in Figure \ref{ftf}.
Note the Wirtinger presentation $\pi_L \cong \langle \alpha_1, \alpha_2 \ | \ \alpha_1 \alpha_2 \alpha_1 =\alpha_2 \alpha_1 \alpha_2 \rangle . $
Then we can easily see
that a correspondence $ \mathcal{C}: \{ \alpha_1, \alpha_2, \alpha_3\} \ra X$
with $\mathcal{C}(\alpha_i )=(x_i ,z_i) \in M \times G $
is an $X$-coloring $ \mathcal{C}$ over $f: \pi_L \ra G$,
if and only if it satisfies the four equations
\begin{equation}\label{eweq22} z_i = f(\alpha_i ), \ \ \ \ \ z_1 z_2 z_1 =z_2 z_1 z_2, \notag
\end{equation}
\begin{equation}\label{eq22} x_3 = x_1 \cdot z_2 +x_2 \cdot (1-z_2 ),
\end{equation}
\begin{equation}\label{eq221} (x_1 -x_2) \cdot (1-z_1 +z_2z_1)=(x_1 -x_2) \cdot (1-z_2 +z_1z_2)=0 . \end{equation}
In particular, Theorem \ref{mainthm2} concerning $\mathrm{Col}_X^{\rm red}(D_f)$ says the isomorphism
$$ H^1(Y_K,\partial Y_K;M ) \cong
\bigl\{ \ x \in M \ \bigr| \ x \cdot (1-z_1 +z_2z_1)=x\cdot (1-z_2 +z_1z_2)=0 \ \bigr\}. $$
Further, given a $G$-invariant linear form $\psi$,
the bilinear form $\mathcal{Q}_{\psi} (\mathcal{C}, \ \mathcal{C}' ) $ is expressed as
$$ \psi \bigl(x_1 -x_2, x_2'(1 -z_2^{-1}) \bigr)+ \psi \bigl(x_2 -x_3, x_3'(1 -z_3^{-1}) \bigr)+\psi \bigl(x_3 -x_1, x_1'(1 -z_1^{-1}) \bigr) \in A ,$$
by definition.
Furthermore, by \eqref{eq22}, the set $ \mathrm{Col}_{X } (D_{f}) $ is generated by the two elements $x_1, x_2$;
Accordingly, it can be seen from \eqref{eq221} that the form $\mathcal{Q}_{\psi}$ is reduced to
\begin{equation}\label{eq2211}\mathcal{Q}_{\psi} \bigl( (x_1, x_2),( x_1', x_2' ) \bigr)= \psi (x_1 -x_2 , \ (x_1' -x_2')\cdot (z_1 - z_1^{-1 }) ) \in A, \end{equation}
where $(x_1^{(')}, x_2^{(')}) \in {\rm Col}_{X^{(')}} (D_{f})\subset (M^{(')})^2 $.
It is worth noting that, if $\psi$ is symmetric, $ \mathcal{Q}_{\psi}$ is zero.
Hence, we should discuss non-symmetric $ \psi$'s.

From the above expressions, we will deal with three concrete representations:
\begin{exa}[cf. Blanchfield pairing]\label{extra1}
Let $f: \pi_L \ra G=\Z= \langle t^{\pm 1 } \rangle$ be the abelianization. 
Then, the equation \eqref{eq221} becomes $(x_1 -x_2)(t^2 -t+1)=0.$
Hence, 
$$ \col_X(D_f) \cong M \oplus \mathrm{Ann}(t^2 -t+1 ).$$
We note that $t^2 -t+1$ is equal to the Alexander polynomial $\Delta_K$ of $K$.
For any elements $x$ and $ x'$ in the annihilator submodule,
the formula \eqref{eq2211} implies
$$ \mathcal{Q}_{\psi} ( x,\ x' )= \psi ( x , \ x' \cdot (2t - 1 ) ).$$
In order to discuss non-trivial cases, for instance, we let $X $ and $ A$ be the PID $\Z [t ]/(t^2 -t+1) $ and $
\psi (y, z)= \bar{y} z $.
Then, the bilinear form $(1-t)\mathcal{Q}_{\psi}$ is summarized to $(1+t) \bar{x} x' $
that is the $(1+t)$-multiple of the Blanchfield pairing, $\bar{x} x' $, as predicted in \cite[Theorem 2.1]{Nos5}.
\end{exa}

\begin{exa}\label{extra22}
Consider $f: \pi_L \ra GL_1(\mathbb{F}[t^{\pm 1 }])= (\mathbb{F}[t^{\pm 1 }])^{\times }$
that sends $\alpha_i$ to $2t $.
Then, the equation \eqref{eq221} becomes $ (x_1 -x_2)(1-2 t+4 t^2)=0$;
hence we should consider $M=A= \mathbb{F}[t]/(1-2 t+4 t^2 ) $, which is not reciprocal.
Furthermore, we can easily see that there is a non-trivial $A$-linear form $ \psi: M^2 \ra A$
if and only if Char$( \mathbb{F}) $=3.
Further, while the associated form $\psi$ of Char$( \mathbb{F}) $=3 is non-degenerate,
\eqref{eq2211} implies degeneracy of the bilinear form
$$\mathcal{Q}_{\psi}(x_1 -x_2 ,x_1' -x_2')=(1-t)(\bar{x_1} -\bar{x_2})(x_1' -x_2').$$
In particular, we have $ (1-t)\mathcal{Q}_{\psi} =0$,
which implies that the form $\mathcal{Q}_{\psi} $ is not always non-degeneracy (cf. Theorem \ref{mainthm14}).
\end{exa}

\begin{exa}[$SL_2$-representations]\label{extra2}
Let $R_{\rm pre}$ be $\Z[s^{\pm 1}, t^{\pm 1}]$ with $\bar{t}=t^{-1}$ and $\bar{s} =s$. 
As considered in the twisted Alexander polynomials,
we will focus on a representation $ f_{\rm pre}: \pi_L \ra SL_2(R_{\rm pre} )$ defined by
$$ f_{\rm pre} (\alpha_1 ) = t \cdot \left(
\begin{array}{cc}
s & 1 \\
0 &s^{-1}
\end{array}
\right), \ \ \ \ \ f_{\rm pre}(\alpha_2 ) = t \cdot \left(
\begin{array}{cc}
s & 0 \\
1- s^2 - s^{-2} & s^{-1}
\end{array}
\right). $$
Here we remark the known fact that the twisted Alexander module is $ \Z[s^{\pm 1}, t^{\pm 1}]/ (t^2 +1)$. 
So, following \S \ref{LFsu32b12r2223}, we shall define $R$ to be $\Z[s^{\pm 1}, t^{\pm 1}]/ (t^2 +1) $
and $M=R^2$, and consider the quotient representation $ f: \pi_L \ra GL_2 ( R)$.
Then, as a solution of \eqref{eq2211}, it can be seen that $ {\rm Col}_X^{\rm red}(D_f )\cong R$ with a basis in
$ {\rm Col}_X^{\rm red}(D_f ) \subset M^2 \cong R^2 \oplus R^2 $ is represented as
$$ \vec{x}= \bigl((0,0), ( ( 1 - s^{-1} t +st )x , \ s t x ) \bigr) , $$
for some $x, \ y, \ z \in R^{\times }$.
Further, we will compute the form $ \mathcal{Q}_{\psi}$,
where an $SL_2(R )$-invariant bilinear form $\psi : (R^2)^{\oplus 2} \ra R$ is the determinant that sends
$ ((a,b),(c,d))$ to $\bar{a}d-\bar{b}c$.
By \eqref{eq2211}, we have
\begin{align}
\lefteqn{ }
\mathcal{Q}_{\psi} \bigl( \vec{x} , \vec{x'} \bigr) &=\det \Bigl( \bar{ \vec{x}}, \ \vec{x'} \bigl( t \left(
\begin{array}{cc}
s & 1 \\
0 &s^{-1}
\end{array}
\right)-t^{-1 }\left(
\begin{array}{cc}
s^{-1} & -1 \\
0 &s
\end{array}
\right) \bigr) \Bigr) \notag \\
&= (s+s^{-1})t \left|
\begin{array}{cc}
( 1- s^{-1} t^{-1} +s t^{-1} ) \bar{x} & ( 1- s^{-1} t +s t )x' \\
st^{-1} \bar{x} & s t x '
\end{array}
\right|
= 2 (s^2 +1 ) \bar{x} x'. \notag
\end{align}
In summary, the concluding point is that the degenerate value $\mathcal{Q}_{\psi } ( \vec{x} , \vec{x'} )$ depends on $s$,
while the twisted module $ \Z[s^{\pm 1}, t^{\pm 1}]/ (t^2 +1)$ does not.

Furthermore, we comment on the non-degeneracy from the viewpoint of Theorem \ref{mainthm14}.
Following from \eqref{aaa},
we shall set up the localized ring $ A= \Z[s^{\pm 1}, t^{\pm 1}][(t-s)(t-s^{-1})^{-1}]/ \mathcal{I}$
with ideal $ \mathcal{I}=(t^2+1)$
and take the resulting representation $ f_{\rm \mathcal{I}}: \pi_L \ra SL_2(A )$.
Then, since we can replace $ x$ by $(t-s)^{-1}x$,
the form $\mathcal{Q}_{\psi} ( \vec{x} , \vec{x'} )$ becomes $2 \bar{x} x' $.
Hence, it is non-degenerate, as indicated in Theorem \ref{mainthm14}.
\end{exa}

\subsection{The figure eight knot}\label{rei.8}
Next, we will compute some $\mathcal{Q}_{\psi}$'s of the figure eight knot.
However, the computation can be done in a similar way to the previous subsection. Thus, we only outline the computation.

Let $D$ be the diagram with arcs as illustrated in Figure \ref{ftf}.
Similarly, we can see
that a correspondence $ \mathcal{C}: \{ \alpha_1, \alpha_2, \alpha_3, \alpha_4\} \ra X$
with $\mathcal{C}(\alpha_i )=(x_i ,z_i) \in M \times G $
is an $X$-coloring $ \mathcal{C}$ over $f: \pi_L \ra G$,
if and only if it satisfies the following equations:
\begin{equation}\label{eq225dd} z_i = f(\alpha_i ), \ \ \ \ \ z_2^{-1} z_1 z_2= z_1^{-1} z_2^{-1}z_1 z_2 z_1^{-1} z_2 z_1 \in G , \end{equation}
\begin{equation}\label{eq225} x_3 = (x_1 -x_2 )\cdot z_2 +x_2 , \ \ \ \ \ \ x_4 = (x_2 -x_1 )\cdot z_1 +x_1, \end{equation}
\begin{equation}\label{eq2215} (x_1 -x_2) \cdot (z_1 + z_2 - 1 )=(x_1 -x_2) \cdot (1-z_2^{-1} ) z_1 z_2 =(x_1 -x_2) \cdot (1-z_1^{-1} ) z_2 z_1 \in M.
\end{equation}
Accordingly, it follows from \eqref{eq225} that the set $ \mathrm{Col}_{X } (D_{f}) $ is generated by $x_1, x_2$;
Given a $G$-invariant bilinear form $\psi$,
it can be seen that the bilinear form $\mathcal{Q}_{\psi}$ is 
expressed as
\begin{equation}\label{eq2211522}\mathcal{Q}_{\psi} \bigl( (x_1, x_2),( x_1', x_2' ) \bigr)= \psi \bigl( x_1 -x_2 , \ (x_1' -x_2') \cdot (1-z_1^{-1} - z_2^{-1} + z_1 z_2^{-1}+ z_2 z_1^{-1})\bigr) \in A, \end{equation}
where $(x_1^{(')}, x_2^{(')}) \in {\rm Col}_{X^{(')}} (D_{f})\subset (M^{(')})^2 $.
We will examine
concrete representations.
\begin{exa}[Elliptic representations]\label{extra222}
Let us set up the situation.
Fix a field $\mathbb{F} $ of characteristic $0.$
Then, we will employ the elliptic representation $f: \pi_1(S^3 \setminus K) \ra SL_2 (\F [t^{\pm 1}] ) $ such that
$$f(\alpha_1)= t \cdot \left(
\begin{array}{cc}
s & 1 \\
0 & s^{-1}
\end{array}
\right), \ \ \ \ \ f(\alpha_2)= t \cdot \left(
\begin{array}{cc}
s & 0 \\
u +1 & s^{-1}
\end{array}
\right),
$$
for some $s,u \in \mathbb{F}^{\times}$ with $\bar{s}=s$ and $ \bar{u }=u$.
We can easily check from \eqref{eq225dd} that $s$ and $u$ must satisfy $ P_{s, u }=0$, where 
$ P_{s, u } := s^2+s^{-2} + u+ u^{-1} -1. $
To state only simple results,
we now assume that $u$ is a quadratic solution of $ P_{s, u } $ (if $u \not{\!\! \in}\ \! \mathbb{F}$, we shall replace $\mathbb{F}$ by a field extension by $ P_{s, u } $).
In addition, we will consider two cases.

\

\noindent
(i) Assume $s+s^{-1} \neq \pm 1$.
Let us consider the canonical action of $ SL_2(\F )$ on $\F^2$.
Then, following \cite{Lin,Wada}, 
the twisted Alexander polynomial $\Delta_f $ associated with $f$ turns out to be
$t^2- 2(s + s^{-1})t+1 $.
Then, 
similar to Example \ref{extra2},
let us define the ring $A$ as $ \mathbb{F}[t]/ ( \Delta_f ) $, and define $M=M'$ as $A^2$ with action.
Then, by the help of computer to solve \eqref{eq2211},
we can verify $ {\rm Col}_X^{\rm red}(D_f )\cong A$ with a basis in
$ {\rm Col}_X^{\rm red}(D_f ) \subset A \cong A^2 \oplus A^2 $ represented as
$$ \vec{x}= \bigl((0,0), ( x , \ \frac{ s^2 - 2 s t + t^2 + s t u }{s- t-s^2t } x ) \bigr) ,$$
for some $x, \ y, \ z \in A^{\times }$.
Further, we will compute the form $ \mathcal{Q}_{\psi}$,
where $\psi : (A^2)^{\oplus 2} \ra A$ is the determinant.
By \eqref{eq2211522}, one can check
$$ \mathcal{Q}_{\psi} ( \vec{x} , \vec{x'} ) = 2 (1 + s^2) (1 - s + s^2) (1 + s + s^2) \bar{x} x . $$
Hence, if $ s^2+1 \neq 0$, this $\mathcal{Q}_{\psi}$ is non-degenerate (cf. the boundary condition in Theorem \ref{mainthm14}).

As an example,
consider the case $\F = \mathbb{C} $ and $(s,u)=(1, (1 +\sqrt{-3})/2 )$. In other ward, $f$ is exactly the holonomy representation arising from the hyperbolic structure of $S^3 \setminus K$;
Then, $ \mathcal{Q}_{\psi}$ is expressed as $ 12 \bar{x} x$.

(ii) On the other hand, we consider the remaining case $s+s^{-1} = \pm 1$.
Then, the associated $H_1(\pi_K; \F [t^{\pm 1}] ^2)$ is annihilated by $t \pm 1$.
Hence, let us define the ring $A$ as $ \mathbb{F}[t]/ ( t \pm 1 ) $, and define $M$ as $A^2$ with action.
Then we can see $ {\rm Col}_X(D_f ) = M^2 \cong A^2 \oplus A^2 $ with basis $ (a,b,c,d)\in A^4$.
Moreover, we can read off from \eqref{eq2211522} that
$$ \mathcal{Q}_{\psi} \bigl((a,b,c,d) , (a',b',c',d') \bigr) = \bar{a} a' + \bar{b} b'. $$
\end{exa}
In addition, in spired by \cite{G}, we discuss $ \mathcal{Q}_{\psi} $ associated with adjoint representations. 
However, as seen in Example \ref{extra222}, we should carefully analyze singular points in the space of representations $f: \pi_K \ra G$.
Thus, we shall focus on generic points such as \eqref{coleqeeee}.
\begin{exa}[Adjoint representations]\label{extra2221113}
Let $G$ and $f$ be as above. Consider the lie algebra $ \mathfrak{g}= \{ B \in \mathrm{Mat}_2(\F )\ | \ \mathrm{Tr}B =0\ \} $ with adjoint action of $SL_2 (\F) $,
and set $M=M':= ( \mathfrak{g }[t^{\pm 1}]/ \mathcal{I})^{2} $ for some ideal $ \mathcal{I} \subset \F [t^{\pm 1 }]$.
Put the Killing 2-form $\psi: \mathfrak{g}^2 \ra \F $ which takes $(X,Y )$ to $\mathrm{Tr}(\bar{X}Y) $.
To state only the simplest result \eqref{coleqeeee222}, let us suppose a generic assumption of the form
\begin{equation}\label{coleqeeee} (u-1)(u+u^{- 1}- 1 )(2u+2u^{- 1}- 1 )(2u+2u^{- 1}- 5 )(u^3 -u^2 -2 u -1 ) \neq 0 .
\end{equation}
Then, 
we can easily compute the twisted Alexander polynomial as
$$ \Delta_f = t^2- (2s^2 + 1 +2s^{-2})t+1 =t^2 +( 2u+2u^{-1}-3 )t+1 .$$
Similarly define the ideal $\mathcal{I}$ to be $ (\Delta_f) $.
Then, as a solution of \eqref{eq2211}, we can show
$ {\rm Col}_X^{\rm red}(D_f )\cong A \ $ with a basis $\vec{x}$ in $\mathfrak{g} \otimes A $:
$$
\left(
\begin{array}{cc}
( 1-t)(u s^2 + s^4 + s^6 + t + s^2 t + u s^4 t) & s(1 - 3 t^2 + 4 u t^2 + t^4 ) \\
N_{s,t,u } & ( t-1 )(u s^2 + s^4 + s^6 + t + s^2 t + u s^4 t)
\end{array}
\right),$$
where the left bottom element $ N_{s,t,u}$ is given by the formula
$$ s^2 - 2 s^4 + (4 s^2 -2) t - 4 s^2 t^2 + (2 + 5 s^2 -
6 u^2 s^2) t^3 - 2 t^4 +
u (s^2 t^4 - 4( 1 + 3 s^2 ) t^3+ 5 s^2 t^2+ 4 (1 - s^2) t -s^2 ).$$
Though the basis is complicated,
the anti-hermitian 2-form $ \mathcal{Q}_{\psi}$ in \eqref{eq2211522} can be reduced to
\begin{equation}\label{coleqeeee222} \mathcal{Q}_{\psi}(\vec{x} , \vec{x'} )= 2 (t - t^{-1}) (u+u^{-1}-1 ) (1-u)(u^3 -u^2 -2 u -1 ) x \bar{x}'. \end{equation}
In contract to the previous examples, this $ \mathcal{Q}_{\psi}$ is parameterized by only $u$, not by the trace $s+s^{-1}$ of $f$.
\end{exa}


\subsection{The $(m,m)$-torus link $T_{m,m}$}\label{Lb12r32}
As a example of computation, we will calculate the bilinear form $\mathcal{Q}_{\psi}$ concerning the $(m,m)$-torus link,
following from Definition \ref{deals3}.
These calculations will be useful in the paper \cite{Nos4}, which
suggests invariants of ``Hurewitz equivalence classes".


Let $L$ be the $(m,m)$-torus link $T_{m,m}$ with $m \geq 2$,
and let $\alpha_1, \dots, \alpha_m$ be the arcs depicted in Figure \ref{ftf}.
Furthermore, let us identity $\alpha_{i+m}$ with $\alpha_i$ of period $m$.
By Wirtinger presentation, we have a presentation of $ \pi_L$ as
$$\langle \ a_1, \dots, a_m \ | \ a_1 \cdots a_m= a_m a_1 a_2 \cdots a_{m-1}= a_{m-1}a_m a_1 \cdots a_{m-2}= \cdots = a_2 \cdots a_m a_1 \ \rangle . $$
In particular, we have a projection $ \mathcal{P}: \pi_L \ra F_{m-1 }$ to the free group of rank $m-1$ subject to $ a_1 \cdots a_m=1$.

Given a homomorphism $f:\pi_L \ra G$ with $f(\alpha_i) \in Z $,
let us discuss $X$-colorings $ \mathcal{C}$ over $f$.
Then, concerning the relation on the $\ell $-th link component, it satisfies the equation
\begin{equation}\label{coleq} \bigl( \cdots (\mathcal{C}(\alpha_\ell ) \lhd \mathcal{C}(\alpha_{\ell+1})) \lhd \cdots \bigr)\lhd \mathcal{C}(\alpha_{\ell+m-1}) =
\mathcal{C}( \alpha_\ell) ,\ \ \ \ \ \ \ {\rm for \ any \ }1 \leq \ell \leq m . \end{equation}
With notation $ \mathcal{C} (\alpha_i):= (x_i,z_i) \in X$,
this equation \eqref{coleq} reduces to a system of linear equations
\begin{equation}\label{aac} ( x_{\ell-1} -x_{\ell} ) + \sum_{ \ell \leq j \leq \ell+m-2 }( x_j -x_{j+1} ) \cdot z_{j+1} z_{j+2} \cdots z_{m +\ell } = 0 \in M , \ \ \ \ \ \mathrm{for \ any \ } 1 \leq \ell \leq m .
\end{equation} 
Conversely, we can easily verify that, if a map $\mathcal{C}: \{ \mbox{arcs of $D$} \} \to X$ satisfies the equation \eqref{aac}, then $\mathcal{C} $ is an $ X$-coloring.
Denoting the left side in \eqref{aac} by $ \Gamma_{f,k}(\vec{x})$, consider a homomorphism
$$ \Gamma_{f}: M^m \lra M^m ; \ \ \ (x_1, \dots, x_m) \longmapsto (\Gamma_{f,1}(\vec{x}) , \dots, \Gamma_{f,m}(\vec{x}) ).$$
To conclude, the set $\mathrm{Col}_X(D_f ) $ coincides with the kernel of $\Gamma_{f} $.

Next, we precisely formulate the resulting bilinear form in Definition \ref{deals3}.
\begin{prop}\label{aa11c} 
Let $f : \pi_1(S^3 \setminus T_{m,m})\ra G$ be as above.
Let $\psi: M \otimes M' \ra A $ be a $G$-invariant bilinear function.
For any $ \ell \in \Z_{>0}$ with $ 1 \leq \ell \leq m $, the bilinear form $\mathcal{Q}_{\psi,\ell}: \mathrm{Ker}(\Gamma_{f} )\otimes \mathrm{Ker}(\Gamma_{f} ')\ \ra A$
takes $ (x_1, \dots, x_m) \otimes(y_1', \dots, y_m') $ to 
\begin{equation}\label{bbbdd}
\sum_{k=1}^{m-1 } \psi \bigl( \sum_{j=1}^{k } (x_{j+\ell -1}-x_{j +\ell })\cdot z_{j+\ell } z_{j+\ell +1} \cdots z_{ k+\ell -1} ,\ y_{k+ \ell }' \cdot (1 - z_{k+ \ell }^{-1}) \bigr) \in A . \end{equation}
\end{prop}
\noindent
The formulae are obtained by direct calculation and definitions.

Finally, under an assumption, we give a 2-dimensional interpretation of the bilinear form.
\begin{prop}\label{aa1133c}
Let $W $ be the complementary space of $S^2$ obtained by removing $m$ disks.
Consider the action of $\pi_1(W ) $ on $M$ induced from that of $\pi_L$ via the above projection $\mathcal{P}: \pi_L \ra F_{m-1}$.
With the notation above,
we assume that $ z_1 \cdots z_m$ are identities in $M$ and $M'$.

Then the diagonal map $ M \ra \Ker ( \Gamma_{\mathbf{z}}) $ is a splitting injection, and
the cokernel is the relative cohomology $H^1(W, \partial W ; M )$.
Further, the bilinear form $\mathcal{Q}_{\psi,\ell} $ coincides with the composite:
$$ H^1( W,\partial W; M ) \otimes H^1( W,\partial W; M ') \xrightarrow{\ \ \smile \ \ }
H^2( W ,\partial W ; M\otimes M' ) \xrightarrow{ \ \langle \psi, \bullet \rangle \circ \langle \bullet , \mu_\ell \rangle \ }A . $$

Furthermore, elements of the image $\mathrm{Im}(\delta^* ) $ and $M_{\rm tri} $ explained in Corollary \ref{mainthm1} are represented by
$ (x_1, \dots, x_m)$ and $ (x, \dots, x) \in M^n$, respectively. Here $x\in M$ and $x_i \in M(1-z_i)$.
\end{prop}
We will give the proof in the end of \S \ref{yy4324}. 
Furthermore, similar to the previous section,
for a concrete representation $ \pi_L \ra G$, we can explicitly compute the bilinear forms.
We refer the reader to \cite{Nos4} for concrete computation from Proposition \ref{aa1133c}.

\section{Proofs of Theorems \ref{mainthm2}}\label{Lbo33222}
We will work out the respective proofs of
Theorem \ref{mainthm2} in \S \ref{yy43} and in \S \ref{yy4324}.
While the statements were described in terms of ordinary cohomology,
the proof will be done via the group cohomology;
In \S \ref{KI11}, we review the relative group (co)homology.
\subsection{Preliminary; Review of relative group homology}\label{KI11}
The relative group homology is a useful method, e.g, for algebraic $K$-theory, secondary characteristic classes and stability problems of group homologies.
As in \cite{BE,T,Zic}, the
relative homology is defined from a projective resolution.
However, we will spell out the relative group (co)homology in {\it non-}homogeneous terms, as follows.

This subsection reviews the definition and properties.
Throughout this subsection, we fix a group $ \Gamma $ and a homomorphism $f : \Gamma \ra G$.
Then, $\Gamma$ acts on the right $G$-module $M$ via $f$.
Let $ C_{n}^{\rm gr }(\Gamma;M ) $ be $M \otimes_{\Z} \Z[\Gamma^n] $.
Define the boundary map $\partial_n( a \otimes (g_1, \dots , g_n)) \in C_{n-1}^{\mathrm{gr}}(\Gamma;M)$ by the formula
\[ a \otimes ( g_2, \dots ,g_{n}) +\!\!\sum_{1 \leq i \leq n-1}\!\! (-1)^i a \otimes ( g_1, \dots ,g_{i-1}, g_{i} g_{i+1}, g_{i+2},\dots , g_n)+(-1)^{n} (a g_n) \otimes ( g_1, \dots , g_{n-1}) .\]
Moreover, we set subgroups $ K_j $ and the inclusions $\iota_j: K_j \hookrightarrow \Gamma $,
where the index $ j$ runs over $ 1 \leq j \leq m $ (possibly, $K_s = K_t $ even if $s \neq t$). 
Then, we can define the complex of 
the mapping cone of $ \iota_j$'s:
More precisely, let us set up the module defined to be
$$ C_n( \Gamma, K_\mathcal{J} ; M ):= C_{n}^{\mathrm{gr}}(\Gamma;M) \oplus \bigl( \bigoplus_{j\in \mathcal{J} } C_{n-1}^{\mathrm{gr}}(K_j ;M) \bigr) $$
and define the differential map on $C_*( \Gamma, K_\mathcal{J}; M ) $ by the formula
$$ \partial_n^{\rm rel}( a, b_1, \dots, b_m):= \bigl(\sum_{j \in \mathcal{J} } \iota_j(b_j) -\partial_n(a), \partial_{n-1}^{ }(b_1), \dots, \partial_{n-1}^{ }(b_m)\bigr) \in C_{n-1} ( \Gamma, K_\mathcal{J} ; M ). $$
Since the square is zero, we can define the relative group homology $H_n( \Gamma, K_{\mathcal{J}}; M ) $.
\begin{rem}\label{clAl221} It is shown \cite[Propositions in \S 1]{T} that, for any $g \in \Gamma$,
the relative homology $ H_*( \Gamma, K_\mathcal{J}; M ) $
is invariant with respect to the change from all the subgroups $K_j$ to $g^{-1} K_j g$.
\end{rem}

Dually, we will discuss the relative cohomology. 
Let us set the cochain group of the form
$$C^n( \Gamma, K_\mathcal{J}; M ):= \mathrm{Map} ( \Gamma^n , M ) \oplus \bigl( \bigoplus_{j } \mathrm{Map }((K_{j})^{n-1} , M ) \bigr) . $$
Furthermore, for $(h,k_1, \dots, k_m ) \in C^n( \Gamma, K_\mathcal{J}; M ) $, let us define
$ \partial^n(h,k_1, \dots, k_m )$ in $ C^{n+1}( \Gamma, K_\mathcal{J}; M )$ by the formula
$$ \partial^n \bigl(h,k_1, \dots, k_m \bigr)( a, b_1, \dots, b_m)= \bigl( h( \partial_{n+1} (a)), \ h (b_1) -k_1(\partial_n(b_1)), \dots,h(b_m) -k_m (\partial_n(b_m))\bigr),$$
where $( a, b_1, \dots, b_m) \in \Gamma^{n+1} \times( K_1)^{n} \times \cdots \times (K_m)^{n} $. Then, we have a complex $ (C^*( \Gamma, K_\mathcal{J}; M ), \partial^*)$, and
can define the cohomology.

As the simplest example, we now observe
the submodule consisting of 1-cocycles.
Let $ \Hom_f (\Gamma , M \rtimes G )$ be the set of group homomorphisms $\Gamma \ra M \rtimes G $ over the homomorphism $f$.
Here the semi-product $M \rtimes G $ is defined by
$$ (a, g) \star (a',g'):=( a \cdot g' + a', \ gg'), \ \ \ \ \mathrm{for} \ \ a,a' \in M, \ \ \ g,g' \in G. $$
Then, as is well-known (see \cite[\S IV. 2]{Bro}), if $K_\mathcal{J} $ is the empty set, the set $ \Hom_f (\Gamma , M \rtimes G )$ 
is identified with the set of group 1-cocycles of $\Gamma $ as follows:
$$ Z^1( \Gamma; M ) \cong \Hom_f (\Gamma , M \rtimes G ); \ \ \ \ \ \ h \longmapsto (\gamma \mapsto ( h(\gamma), f(\gamma))) .$$
Further, concerning the relative cohomology, from the definition, we can easily characterize the first cohomology as follows:
\begin{lem}\label{clAl1}
The submodule of 1-cocycles, $Z^1( \Gamma, K_\mathcal{J}; M )$, is identified with the following:
$$ \{ \ (\widetilde{f} , y_1, \dots, y_m ) \in \Hom_f ( \Gamma, M \rtimes G) \oplus M^m \ | \ \ \widetilde{f} (h_j) = ( y_j - y_j \cdot h_j, \ f_j( h_j) ) , \ \ \mathrm{for \ any } \ h_j \in K_j. \ \} $$

Moreover, the image of $\partial^1$, i.e., $B^1( \Gamma, K_\mathcal{J}; M )$, is equal
to the subset $\{ ( \tilde{f}_a, a, \dots, a)\}_{a \in M}. $
Here, for $a \in M,$ this
$ \tilde{f}_a : \Gamma \ra M \rtimes G $ is defined as a map which sends $\gamma $ to $( a - a \cdot \gamma, \ f (\gamma ))$.
In particular, if $ K_\mathcal{J} $ is not empty, $B^1( \Gamma, K_\mathcal{J}; M )$ is a direct summand of $ Z^1( \Gamma, K_\mathcal{J}; M )$.
\end{lem}

Finally, we will formulate explicitly the cup product on $ C^p( \Gamma, K_\mathcal{J} ; M )$ and the Kronecker product.
When $ K_\mathcal{J} $ is the empty set, we define the product of $u \in C^p( \Gamma; M )$ and $v \in C^{q}( \Gamma; M' )$ to be 
the element $u \smile v \in C^{p+q} ( \Gamma; M \otimes M')$ given by
$$ ( u \smile v) ( g_1 ,\dots, g_{p+q}):= (-1)^{pq} \bigl( u (g_1 ,\dots, g_{p} ) g_{p+1} \cdots g_{p+q} \bigr) \otimes v (g_{p+1} ,\dots, g_{p+q} ) .$$
Further, if $ K_\mathcal{J} $ is not empty, for two elements $(f,k_1, \dots, k_m )\in C^p( \Gamma, K_\mathcal{J} ; M )$ and
$(f',k'_1, \dots, k'_m )\in C^{q}( \Gamma, K_\mathcal{J} ; M' )$, let us define {\it the cup product} to be the formula
$$ ( f \smile f', \ k_1\smile f', \dots, \ k_m\smile f') \in C^{p+q}( \Gamma, K_\mathcal{J} ; M \otimes M'). $$
We can easily see that this formula descends to a bilinear map, by passage to cohomology,
$$ \smile : H^p( \Gamma, K_\mathcal{J} ; M ) \otimes H^{q}( \Gamma, K_\mathcal{J} ; M' )\lra H^{p+q}( \Gamma, K_\mathcal{J} ; M \otimes M').$$
Then the graded commutativity holds: for any $u \in H^{p} ( \Gamma, K_\mathcal{J} ; M )$ and $v \in H^{q}( \Gamma, K_\mathcal{J} ; M' )$,
we have $u \smile v =(-1)^{pq } \tau (v \smile u) $,
where $ \tau: M \otimes M' \ra M' \otimes M$ is the canonical isomorphism.
Furthermore, for $( a, b_1, \dots, b_m) \in \Gamma^{n} \times K_1^{n-1} \times \cdots \times K_m^{n-1} $,
consider the evaluation defined by
$$ \langle (f,k_1, \dots, k_m ) , ( a, b_1, \dots, b_m) \rangle:= f(a)+k_1(b_1)+ \cdots +k_m(b_m)\in M.$$
Then it can be seen that the formula induces $ \langle , \rangle : H^n( \Gamma, K_\mathcal{J} ; M )\otimes H_n( \Gamma, K_\mathcal{J} ; A )\ra H_0(\Gamma ; M) . $
Here we can replace $ H_0(\Gamma ; M) $ by the coinvariant $M_\Gamma= M/ \{ ( a -a \cdot g)\} _{a \in M, g \in \Gamma}$.
\begin{rem}\label{clAl221}
We will give a topological description of the above definitions without proofs (For the proof see \cite{BE} or \cite[\S 3]{Zic}).
Consider the Eilenberg-MacLane spaces of $ \Gamma$ and of $K_j$, and
the map $ (\iota_j)_* : K(K_j,1 ) \ra K( \Gamma,1 ) $ induced by the inclusions.
Then the relative homology $H_n( \Gamma, K_\mathcal{J} ; M ) $ is isomorphic to the homology of the mapping cone of $ \sqcup_j K(K_j,1 )\ra K(G,1 )$
with local coefficients.
Further, the cup product $\smile$ and the Kronecker product $ \langle , \rangle$ above coincide with those on
the usual singular (co)homology groups (up to signs \footnote{See \cite[\S\S 1-2 ]{BE} for details.}).
In particular, we mention the knot case $\# L =1$.
Since the complementary space $Y_L = S^3 \setminus L$ is an Eilenberg-MacLane space, we
have an isomorphism $ H^*( \pi_1 (Y_L ), \pi_1( \partial Y_L) ;M) \cong H^*( Y_L, \partial Y_L ; M) $.

More generally, we comment on the case $ \# L \geq 1$.
We let $ \Gamma$ be $\pi_1(Y_L)$ and let $ K_\ell ( \cong \Z^2 ) $ be the abelian subgroup of $\pi_1(Y_L) $ arising from
the $\ell $-th boundary. 
Denote the family $\mathcal{K}:=\{K_\ell \}_{\ell \leq \# L}$ by $\partial \pi_1 (Y_L)$,
and consider the inclusion pair
$$\iota_Y : \bigl( Y_L ,\ \partial Y_L \bigr) \ra \bigl( K(\pi_1 (Y_L),1 ) , \ K(\partial \pi_1 (Y_L),1 ) \bigr) $$
obtained by attaching cells to kill the higher homotopy group.
Then, we have a commutative diagram:
$${\normalsize
\xymatrix{
H^1( \pi_1 (Y_L), \partial \pi_1(Y_L) ; M)^{\otimes n} \ar[r]^{\smile}\ar[d]_{\cong }^{\iota_Y^*} & H^n( \pi_1 (Y_L), \partial \pi_1(Y_L) ; M^n ) \ar[rr]^{\ \ \ \ \ \ \ \ \ \ \langle \bullet, (\iota_Y )_*(\mu) \rangle }
\ar[d]^{\iota_Y^* }& & \ \ (M^{\otimes n})_{\pi_1(Y_L)} \ar@{=}[d] \ \ & \\
H^1( Y_L ,\partial Y_L; M)^{\otimes n} \ar[r]^{\smile} & H^n( Y_L ,\partial Y_L; M^n) \ar[rr]^{\ \ \ \ \ \ \ \ (-1)^n\langle \bullet, \mu \rangle } & & \ \ (M^{\otimes n})_{\pi_1(Y_L)}.
}}
$$
Here, the left $\iota_Y^*$ is an isomorphism from the definition of $ \iota_Y $.
In conclusion, as a result of the diagram, to prove Theorem \ref{mainthm2} on the bottom arrows,
we may focus on only the group (co)homologies on the upper one such as the next subsection, in what follows.
\end{rem}


\subsection{Proof for the isomorphism \eqref{g21gg33}.}\label{yy43}
This subsection gives the proof of the isomorphism \eqref{g21gg33} in Theorem \ref{mainthm2}, and Corollary \ref{mainthm1}.
For this, 
we now prepare terminology throughout this section:
Let $D$ be a diagram of a link $L$ and let $\Gamma$ be $ \pi_L$. 
In addition,
we fix an arc $\gamma_{\ell}$
from each link component $\ell$ of $L$, and
consider the circular path $\mathcal{P}_{\ell}$ starting from $\gamma_{\ell} $ (see Figure \ref{ezu}).
Further, for $j\geq 2$, we denote by
$\alpha_{\ell, j}$ the $j$-th arc on $\mathcal{P}_{\ell} $,
and do by $\beta_{\ell, j}$ the arc
that divides the arcs $\alpha_{\ell, j-1 } $ and $\alpha_{\ell, j}$. 
Considering the meridian $\mathfrak{m}_{\ell,j}^{ \epsilon_j} \in \pi_1( S^3 \setminus L )$ associated with the arc $ \beta_{\ell, j} $,
we here define the longitude $ \mathfrak{l}_{\ell}$ to be
\begin{equation}\label{189}
\mathfrak{l}_{\ell} := \mathfrak{m}_{\ell,1}^{ \epsilon_1} \mathfrak{m}_{\ell,2}^{ \epsilon_2} \cdots \mathfrak{m}_{\ell, N_\ell }^{ \epsilon_{N_\ell}} \in \pi_1(S^3 \setminus L),
\end{equation}
where 
$\epsilon_i \in \{ \pm 1\} $ is the sign of the crossing between $ \alpha_{\ell,i } $ and $\beta_{\ell,i }$ with $i>1,$
and $ \epsilon_1 =+1$.
Considering the subgroup $\partial_{\ell} \pi_L \cong \Z^2 $ generated by the meridian-longitude pair $ (\mathfrak{m}_{\ell} , \mathfrak{l}_{\ell} )$,
the union $ \partial_{1} \pi_L \sqcup \cdots \sqcup \partial_{\# L} \pi_L$ coincides with the family $ \partial \pi_L $ mentioned in Remark \ref{clAl221}.

\begin{figure}[htpb]
\begin{center}
\begin{picture}(50,74)
\put(-104,40){\large $\beta_{\ell, 1 } $}
\put(187,51){\Large $\mathcal{P}_\ell $}
\put(-132,-1){ \includegraphics[clip,width=11.0cm,height=2.6cm]{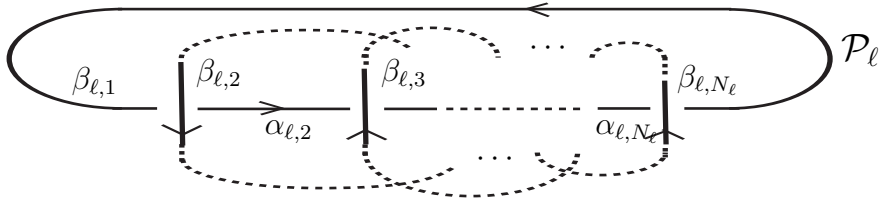}}

\put(-31,23){\large $\alpha_{\ell, 2 }$}
\put(94,23){\large $\alpha_{\ell, N_{\ell}} $}
\put(49,11){\large $\cdots $}

\put(-57,43){\large $\beta_{\ell, 2 } $}
\put(13,43){\large $\beta_{\ell, 3} $}
\put(125,41){\large $\beta_{\ell, N_{\ell}} $}
\put(69,53){\large $\cdots $}
\end{picture}
\end{center}
\vskip -0.95pc
\caption{\label{ezu} The longitude $\ell_j$ and arcs $\alpha$'s and $\beta$'s in the diagram $D$. Here $\gamma_{\ell}= \alpha_{\ell ,1 }= \beta_{\ell, 1 }. $}
\end{figure}

\begin{proof}[Proof of the isomorphisms \eqref{g21gg33}]
First, we will construct a map in \eqref{11222}.
Given a $G$-module $M$,
set up a map
\begin{equation}\label{g21gg} \kappa : M \times G \lra M \rtimes G; \ \ \ \ (m,g) \longmapsto (m \cdot g - m, \ g ). \end{equation}
Further, for an $X$-coloring $\mathcal{C}$ over $f$,
consider a map 
$ \tilde{f}_{\mathcal{C}} : \{ {\rm arcs \ of \ }D \} \ra M \rtimes G $
which takes $\gamma$ to $\kappa \bigl( \mathcal{C} (\gamma) \bigr) $.
Then, we verify from Wirtinger presentation that this $ \tilde{f}_{\mathcal{C}}$ defines a group homomorphism $\pi_L \ra M \rtimes G$ over $f$.
Hence, 
we obtain a map
\begin{equation}\label{11222} \Omega: \mathrm{Col}_X(D_{f}) \lra \Hom (\pi_L , M \rtimes G) \times (X^{\# L}) ; \ \ \ \mathcal{C} \longmapsto ( \tilde{f}_{\mathcal{C}}, \ \mathcal{C}( \gamma_1), \dots, \mathcal{C}( \gamma_{\# L} ) ). \end{equation}




We claim tha such an $ \tilde{f}_{\mathcal{C}}$ uniquely admits
$a_{\ell} \in M$
satisfying the two identities
\begin{equation}\label{1v24}
\ \widetilde{f}(\mathfrak{m}_{\ell } ) =(a_{\ell} - a_{\ell} \cdot f(\mathfrak{m}_{\ell } ), \ f(\mathfrak{m}_{\ell } ))
\ \ \ \ \ \ \ \ 
\widetilde{f}(\mathfrak{l}_{\ell } ) =(a_{\ell} - a_{\ell} \cdot f(\mathfrak{l}_{\ell } ), \ f(\mathfrak{l}_{\ell } )) \in M \rtimes G.
\end{equation}
with respect to $1 \leq \ell \leq \# L.$ Fix notation $\mathcal{C}(\beta_{\ell,j})=(y_j, z_j)\in M \times G $.
The first one is obvious from the definition of $\kappa$ with $a_{\ell}=y_1$.
We shall show the second:
by the coloring condition between $\alpha_{\ell. N_{\ell}}$ and $\beta_{\ell. N_{\ell}}$, we have two equations;
\begin{equation}\label{deq221}
z_1 z_2^{\epsilon_2} \cdots z_{N_{\ell}}^{\epsilon_{N_{\ell}}} =z_2^{\epsilon_2} \cdots z_{N_{\ell}}^{\epsilon_{N_{\ell}}} z_1 , \ \ \ \ \
y_1 z_2 ^{\epsilon_{2}} \cdots z_{N_{\ell}} ^{\epsilon_{N_{\ell}}} +
\sum_{k=2}^{N_{\ell}-1 } y_k (1-z_{k+1}^{\epsilon_{k+1}} )\cdot z_{k+2 }^{\epsilon_{k+2}} z_{k+3 }^{\epsilon_{k+3}} \cdots z_{N_{\ell}}^{\epsilon_{N_{\ell}}} = y_1 .
\end{equation}
By \eqref{g21gg} and \eqref{189} that $\widetilde{f}(\mathfrak{l}_{\ell } ) $ is expressed as
$\sum_{k=1}^{N_{\ell} } y_k (1-z_{k+1}^{\epsilon_{k+1}} )\cdot z_{k+2 }^{\epsilon_{k+2}} z_{k+3 }^{\epsilon_{k+3}} \cdots z_{N_{\ell}}^{\epsilon_{N_{\ell}}} $.
Hence, comparing carefully it with \eqref{deq221} gives the second one in \eqref{1v24}

Since \eqref{1v24} coincides with exactly the 1-cocycle condition by Lemma \ref{clAl1},
the map $\Omega$ is reduced to $\mathrm{Col}_X(D_{f}) \ra Z^1( \pi_L, \partial \pi_L;M)$.
We will construct the inverse mapping as follows.
For this, notice the equality
$$ \kappa (a \cdot h +b, h^{-1}gh) = (b,h)^{-1} \cdot \kappa (a,g) \cdot (b,h) \in M \rtimes G $$
from the definitions, and notice that any meridian $ \mathfrak{m}_{\ell, j} $ in $\pi_L $ is conjugate to the $ \mathfrak{m}_{\ell_j }$ on the $\ell_j$-th component for some $\ell_j$:
In other ward, we can choose $ h_j \in \pi_L $ with $\mathfrak{m}_{\ell, j} = h_j^{-1}\mathfrak{m}_{\ell_j} h_j$.
To summarize, given an $\tilde{f}$ in $ Z^1( \pi_L, \partial \pi_L;M)$,
we define a map $\mathcal{C}_{\tilde{f}} : \{ \mathrm{arc \ of \ }D \} \ra X$ by $ \mathcal{C}_{\tilde{f}} (\mathfrak{m}_{\ell, j} ) = ( a_{\ell} \cdot h_j +b_j , f (\mathfrak{m}_{\ell, j} ) )$
where $b_j \in M$ is defined from $\tilde{f}( \mathfrak{m}_{\ell, j} )=(b_j, f (\mathfrak{m}_{\ell, j} ))$.
Then, we can easily see that $ \mathcal{C}_{\tilde{f}}$ is an $X$-coloring, and this construction gives the desired inverse mapping.

To summarize,
we have the isomorphism $\mathrm{Col}_X(D_{f}) \lra Z^1( \pi_L, \partial \pi_L;M)= H^1( \pi_L, \partial \pi_L;M) \oplus M$.
Furthermore, Lemma \ref{clAl1} again says that the summand $ X_{\rm diag} \cap \mathrm{Col}_X(D_{f}) \cong M$ is exactly $B^1(\pi_L, \partial \pi_L;M )$.
Hence, by the definition of $\mathrm{Col}_{X }^{\rm red} (D_{f}) $ in \eqref{skew24592}, the map $\Omega$ ensures the desired $\mathrm{Col}_{X }^{\rm red} (D_{f}) \cong H^1(Y_L , \ \partial Y_L ;M )$.
\end{proof}

\begin{proof}[Proof of Corollary \ref{mainthm1}]
We will show the required isomorphism $ H^1(\pi_L , \partial \pi_L ; M) \cong H^1(\pi_L ; M)$ from
the bijective assumption of $\mathrm{id} - f (\mathfrak{m}_{\ell}): M \ra M$.
For this, it is enough to construct an inverse mapping of the projection $Z^1(\pi_L , \partial \pi_L ; M) \ra Z^1(\pi_L ; M) $.
Let $\tilde{f}:\pi_L \ra M \rtimes G$ be any homomorphism over $f$ as being in $Z^1(\pi_L ; M) $.
Choose some $b_{\ell}$ and $ c_{\ell} \in M$ with
$$\widetilde{f}(\mathfrak{m}_{\ell })= ( b_{\ell}, \ f(\mathfrak{m}_{\ell })) , \ \ \ \ \ \ \ \ \widetilde{f}(\mathfrak{l}_{\ell } )= ( c_{\ell}, \ f(\mathfrak{l}_{\ell })) \in M \rtimes G .$$
Since the pair $ ( \mathfrak{m }_{\ell },\mathfrak{l}_{\ell }) $ commutes in $\pi_L$,
we have $ \widetilde{f}(\mathfrak{m}_{\ell }) \widetilde{f}(\mathfrak{l}_{\ell } )=\widetilde{f}(\mathfrak{l}_{\ell }) \widetilde{f}(\mathfrak{m}_{\ell } )$, which reduces to
$$ \bigl( \ c_{\ell}- c_{\ell}f(\mathfrak{m}_{\ell } ) - b_{\ell}+ b_{\ell}f(\mathfrak{l}_{\ell } \bigr), \ \ f(\mathfrak{l}_{\ell })^{-1} f(\mathfrak{m}_{\ell } )^{-1}f(\mathfrak{l}_{\ell }) f(\mathfrak{m}_{\ell } ) \ \bigr)= (0, 1_g) \in M \rtimes G. $$
Setting $a_{\ell}= b_{\ell} (\mathrm{id} - f(\mathfrak{m}_{\ell }) )^{-1} $
by assumption,
the reduced equality implies $ c_{\ell}= a_{\ell} (\mathrm{id} - f(\mathfrak{l}_{\ell }) )$.
Hence, the correspondence $ \widetilde{f} \mapsto (\widetilde{f}, a_1, \dots, a_{\# L})$
gives rise to the desired inverse mapping.

Incidentally, the vanishing $\mathrm{Im}(\delta^*)$ is obtained from $H ^1( \partial \pi_L ; M)=0$.
\end{proof} 

\subsection{Proofs of Theorem \ref{mainthm2} \ and Proposition \ref{aa1133c}}\label{yy4324}
We turn into proving Theorem \ref{mainthm2} and Proposition \ref{aa1133c}.
The proof can be outlined as concrete computations of
the bilinear form $\mathcal{Q}_{\psi, \ell}$ and of the cup product
in turn.
The point here is to describe explicitly the 2-cycle $\mu_{\ell }$ in Lemmas \ref{2gs22} and \ref{2g2s22}.

To accomplish the outline,
one will compute $\mathcal{Q}_{\psi, \ell}$. 
Recall the arc $\beta_{\ell,j}$ explained in Figure \ref{ezu}.
For two $X$-colorings $ \mathcal{C}$ and $ \mathcal{C}'$,
we further employ notation $\mathcal{C}(\beta_{\ell,j})=(y_j, z_j)\in M \times G $ and $\mathcal{C}'(\beta_{\ell,j})=(y_j', z_j) \in M' \times G$.
Then, one can easily verifies that the value $\mathcal{Q}_{\psi,\ell } (\mathcal{C}, \mathcal{C}' )$ is, from the definition, formulated as
\begin{equation}\label{deqqg}
\sum_{k=1}^{N_{\ell}-1 } \psi \bigl( y_1 z_2^{\epsilon_{2}} \cdots z_k^{\epsilon_{k}} - y_{k+1} + \sum_{j=2}^{k} y_{j} (1-z_{j}^{\epsilon_{j}} ) z_{j+1 }^{\epsilon_{j+1}} z_{j+2 }^{\epsilon_{j+2}} \cdots z_{ k }^{\epsilon_{k}} ,\ y_{k + 1 }' \cdot (1 -z_{k +1}^{-\epsilon_{k+1}}) \bigr) ,
\end{equation}
where the second sigma with $k=1$ means zero (cf. \eqref{bbbdd} as the case that all $\epsilon_j=1$ and $\ell=1$).

On the other hand, let us compute the cup products (Lemmas \ref{2gs} and \ref{2g2s}).
To this end, we now introduce a 2-cycle.
Consider the abelian subgroup $ \langle \mathfrak{m}_{\ell, j} \rangle \cong \Z $ generated by
the meridian $\mathfrak{m}_{\ell, j}$ with respect to the arc $\beta_{\ell, j}$.
Then, we write $ \mathfrak{M}_{\ell}$ for the disjoint union
$ \langle \mathfrak{m}_{\ell, 1} \rangle \sqcup \langle \mathfrak{m}_{\ell, 2} \rangle \sqcup \cdots \sqcup \langle \mathfrak{m}_{\ell, N_{\ell}} \rangle $,
and do $ \mathfrak{M}_{\rm arc }$ for the whole union
$ \sqcup_{\gamma } \langle \mathfrak{m}_{\gamma} \rangle$
running over every arcs $\gamma $.
Let us define an element $ \hat{\mu}_\ell^{\rm pre} $ in the relative complex $ C_2(\pi_L , \partial \pi_L \sqcup \mathfrak{M}_{\ell};
\Z)$ with trivial coefficients to be
$$ ((1 , 1), \mathfrak{l}_{\ell}) +\sum_{k=1}^{N_{\ell}-1} ( (\mathfrak{m}_{\ell, 1}^{\epsilon_1}
\cdots \mathfrak{m}_{\ell, k}^{\epsilon_k},
\mathfrak{m}_{\ell, k+1}^{\epsilon_{k+1}}) , 1) - \sum_{k=1}^{N_{\ell}} ( (1,1) , \mathfrak{m}_{\ell, k}^{\epsilon_{k}}) .$$
Here, the last term
has only the non-trivial $ (k +\# L+1)$-th component $\mathfrak{m}_{\ell, k}^{\epsilon_k }$.
Then we can easily see that $\hat{\mu}_\ell^{\rm pre }$ is a 2-cycle.
Moreover, 
it is easy to verify the following lemma:
\begin{lem}\label{2gs}
Take the inclusion pair $\iota_Y$ in Remark \ref{clAl221}, and
the relative composite map 
$$r_Y:\bigl( K(\pi_1 (Y),1 ) , \ K(\partial \pi_1 (Y),1 ) \bigr) \lra \bigl( K(\pi_1 (Y),1 ) , \ K(\partial \pi_1 (Y) \sqcup \mathfrak{M}_{\ell},1 ) \bigr)$$
induced from the inclusions-pair $ (\pi_1 (Y) , \ \partial \pi_1 (Y) )\ra (\pi_1 (Y) , \ \partial \pi_1 (Y) \sqcup \mathfrak{M}_{\ell})$.
Consider the $\ell$-th 2-cycle $ \mu_{\ell} \in H_2( Y_L ,\partial Y_L;\Z)\cong \Z^{\# L} $ as before.
Then $(r_Y \circ \iota_Y )_*(\mu_{\ell})= \hat{\mu}^{\rm pre }_{\ell}. $
\end{lem}
\begin{rem}\label{2gs22}
In some case with $\# L >1$, the homology class $\hat{\mu}^{\rm pre }_{\ell}$ vanishes. For example, if $L$ is the Hopf link,
$Y_L$ is homotopic to the one of the boundary tori $S^1 \times S^1$.
Hence, we can easily verify that, for any local system $M$, the second homology $ H_2(\pi_1 (Y),\ \partial \pi_1 (Y) ;M)$ vanishes.

In comparison with Proposition \ref{aa11c},
we claim that every bilinear form $\mathcal{Q}_{\psi}$ of the Hopf link $L$ is trivial.
Actually, for the diagram $D$ with two arcs $\alpha_1$ and $ \alpha_2$, the formula \eqref{aac} becomes $ (x_1-x_2) (1-z_1)= (x_1-x_2)(1 -z_2 ) =0$,
and the formulation \eqref{bbbdd} on $\mathcal{Q}_{\psi}$ is reduced to $ \psi(x_1-x_2,\ x_2 (1-z_\ell^{-1}))= 0$ by the $G$-invariance.
\end{rem}

Next, for $\ell \leq \# L,$ we will set up a homomorphism between the sets of 1-cocycles
$$ \zeta_{\ell}: Z^1(\pi_L , \partial \pi_L ;M) \lra Z^1(\pi_L , \partial \pi_L \sqcup \mathfrak{M}_\ell;M) $$
as follows. Recall the terminology $b_j \in M $ in the proof in \S \ref{yy43}.
By Lemma \ref{clAl1}, every element of $Z^1(\pi_L , \partial \pi_L ;M) $
can be represented by a homomorphism $ \widetilde{f}$ with $(a_1, \dots, a_{\#L}) \in M^{\#L}$
satisfying $\widetilde{f}( \mathfrak{m}_{\ell,1})= \kappa (a_\ell , f(\mathfrak{m}_{\ell,1})) .$
Hence, the correspondence
$$ ( \widetilde{f}, a_1, \dots, a_{\# L}) \longmapsto (\widetilde{f}, a_1, \dots, a_{\# L}, b_1, \dots, b_{N_{\ell}} ) $$
yields the desired homomorphism $\zeta_{\ell} $.
In addition, by iterating the process, we can similarly obtain a homomorphism
$\zeta: Z^1(\pi_L , \partial \pi_L ;M) \ra Z^1(\pi_L , \partial \pi_L \sqcup \mathfrak{M}_{\rm arc} ;M)$.

We will use these $\zeta$ and $\zeta_{\ell}$ to recover the bilinear form $\mathcal{Q}_{\psi}$ from some cup product:
\begin{lem}\label{2g2s}
For any two colorings $\mathcal{C}$ and $\mathcal{C}' $, consider the cup product of the form
$$ \mathcal{K}_{\mathcal{C}, \mathcal{C}'}:=
\bigl( \zeta_{\ell} \circ \Omega(\mathcal{C}) \bigr) \smile \bigl( \zeta'_{\ell} \circ \Omega'(\mathcal{C}')\bigr) \in Z^2 ( \pi_L , \partial \pi_L \sqcup \mathfrak{M}_{\ell};M \otimes M'). $$
Then, the pairing $ \psi ( \langle \mathcal{K}_{\mathcal{C}, \mathcal{C}'},\ \hat{\mu}^{\rm pre }_{\ell} \rangle) $ is equal to
the value $\mathcal{Q}_{\psi}(\mathcal{C}, \mathcal{C}') . $
\end{lem}
\begin{proof}
By Lemma \ref{clAl1}, the composite $ \zeta_{\ell}^{(')} \circ \Omega^{(')}(\mathcal{C}^{(')})$ forms
$ (\widetilde{f}^{(')}, a_1^{(')}, \dots, a_{\# L}^{(')}, b_1^{(')}, \dots, b_{N_{\ell}}^{(')} ) $.
Then
the cup product $ \mathcal{K}_{\mathcal{C}, \mathcal{C}'}$ is, by definition,
formulated as
$$ ( \widetilde{f} \smile \widetilde{f}' , \ a_1 \otimes \widetilde{f}', \dots, a_{\# L}\otimes \widetilde{f}', \ b_1\otimes \widetilde{f}', \dots, \ b_{N_{\ell}} \otimes \widetilde{f}' ) . $$
Write $ \hat{\mu}_{(1)}, \ \hat{\mu}_{(2)}$ and $ \ \hat{\mu}_{(3)} $ for the first, second and third term in $\hat{\mu}^{\rm pre }_{\ell} $, respectively.
We will compute the pairings $ \langle \mathcal{K}_{\mathcal{C}, \mathcal{C}'},\ \hat{\mu}_{(i)} \rangle$.
Note from the definitions that the third term $\langle \mathcal{K}_{\mathcal{C}, \mathcal{C}'},\ \hat{\mu}_{(3)} \rangle $ is
$- \sum_{k=1}^{N_{\ell} } b_{k}\otimes b_{k}'(1-z_{k}^{-\epsilon_{k}} ) $.
Next, the first one
$\langle \mathcal{K}_{\mathcal{C}, \mathcal{C}'},\ \hat{\mu}_{(1)} \rangle = b_1 \otimes \tilde{ f} (\mathfrak{l}_{\ell })$ is written in
$$\bigl( b_1 z_1^{\epsilon_1} \cdots z_{N_{\ell}}^{\epsilon_{N_{\ell}}}\bigr)\otimes \bigl(
\sum_{k=1}^{N_{\ell}} b_{k}' (1-z_{k}^{\epsilon_{k}} ) z_{k+1 }^{\epsilon_{k+1}} z_{k+2 }^{\epsilon_{k+2}} \cdots z_{N_{\ell}}^{\epsilon_{N_{\ell}}} \bigr)
= - \sum_{k=1}^{N_{\ell}} \bigl( b_1 z_1^{\epsilon_1} \cdots z_{k -1 }^{\epsilon_{k-1 }}\bigr)\otimes \bigl(
b_{k }' (1-z_{k}^{ -\epsilon_{k}} ) \bigr).$$
Finally, we now compute the second term as
\[\psi \langle \mathcal{K}_{\mathcal{C}, \mathcal{C}'},\ \hat{\mu}_{(2)} \rangle =
\sum_{k=1}^{N_{\ell}-1 } \psi \bigl( \sum_{j=1 }^k b_j(1-z_{j}^{\epsilon_j}) z_{j+1}^{\epsilon_{j+1}}
\cdots z_{k+1}^{\epsilon_{k+1}} ,\ b_{k + 1 }' \cdot (1 -z_{k +1}^{\epsilon_{k+1}}) \bigr)
\]
\[ = \sum_{k=1}^{N_{\ell}-1 } \psi \bigl(
-b_1 z_1^{\epsilon_1} \cdots z_{N_{\ell}}^{\epsilon_{N_{\ell}}}+
b_1 z_2^{\epsilon_2} \cdots z_{N_{\ell}}^{\epsilon_{N_{\ell}}} +
\sum_{j=2}^k b_j(1-z_{j}^{\epsilon_j}) z_{j+1} ^{\epsilon_{j+1}}
\cdots z_{k}^{\epsilon_{k}} ,\ b_{k + 1 }' \cdot (1 -z_{k +1}^{- \epsilon_{k+1}}) \bigr) .\]
Here, notice that this first term equals $\psi( b_1 \otimes \tilde{ f} (\mathfrak{l}_{\ell }) +
( b_1 , b_{1}' -b'_{1} z_{1}^{-\epsilon_{1}}))$.
To summarize, comparing the sum $ \langle \mathcal{K}_{\mathcal{C}, \mathcal{C}'},\ \hat{\mu}_{(1)}+ \hat{\mu}_{(2)}+ \hat{\mu}_{(3)} \rangle $
with the formula \eqref{deqqg} of $\mathcal{Q}_{\psi}(\mathcal{C}, \mathcal{C}') $
immediately accounts for the desired equality.
\end{proof}
As the next step, let us reduce the 2-cycle $\hat{\mu}_{\ell}^{\rm pre}$ to a 2-cycle in $ C_2(\pi_L , \partial \pi_L ;\Z) $:
\begin{lem}\label{2g2s22}
Take two arcs $\alpha$ and $ \beta$ from the same link component, and
put the associated meridians $ \mathfrak{m}_{\alpha}$ and $ \mathfrak{m}_{ \gamma} \in \pi_L$.
Then there exists $ \nu^{\rm pre}_{\alpha, \gamma } \in C_2(\pi_L;\Z ) $ such that the 2-chain $ \nu_{\alpha, \gamma} \in C_2(\pi_L , \partial \pi_L \sqcup \mathfrak{M}_{\rm arc};\Z ) $ of the form
\begin{equation}\label{aa6}
(\nu_{\alpha, \gamma }^{\rm pre}, 1) -( (1,1), \mathfrak{m}_{\alpha} )- ( (1,1), \mathfrak{m}_{ \gamma} )
\end{equation}
is a 2-cycle
and that the following pairing is zero:
\begin{equation}\label{010101} \langle \bigl( \zeta \circ \Omega (\mathcal{C}) \bigr) \smile \bigl( \zeta' \circ \Omega' (\mathcal{C}')\bigr) , \ \nu_{\alpha, \gamma }
\rangle \in (M \otimes M')_{\pi_L }. \end{equation}
\end{lem}
\begin{proof}
Without loss of generality, we may assume that 
$\alpha$ is next to $ \beta $ and
an arc $\beta $ separates $\alpha$ from $ \gamma$ (see Figure \ref{koutenpn}).
Let us denote by $\epsilon$ the sign of the crossing, and
define $\nu^{\rm pre}_{\alpha, \gamma }$ to be 
$$ (\mathfrak{m}_{\alpha}, \mathfrak{m}_{\beta}^{\epsilon}) + (\mathfrak{m}_{\beta}^{-\epsilon} ,
\mathfrak{m}_{\alpha} \mathfrak{m}_{\beta}^{\epsilon})-(\mathfrak{m}_{\beta}^{\epsilon},\mathfrak{m}_{\beta}^{-\epsilon} )-(1,1) \in C_2(\pi_L;\Z ). $$
Since $ \partial_2(\nu^{\rm pre}_{\alpha, \gamma})= (\mathfrak{m}_{\alpha})+( \mathfrak{m}_{\beta}^{-\epsilon}\mathfrak{m}_{\alpha} \mathfrak{m}_{\beta}^{\epsilon})=
(\mathfrak{m}_{\alpha})+(\mathfrak{m}_{ \gamma}) $,
the 2-chain $ \nu^{\rm pre}_{\alpha, \gamma }$ in \eqref{aa6} is a 2-cycle.
In addition, after a computation, we can
verify that the pairing \eqref{010101} is zero.
\end{proof}
Using the above preparation, we now prove Theorem \ref{mainthm2} and Proposition \ref{aa1133c}.
\begin{proof}[Proof of Theorem \ref{mainthm2}]
To start, we will formulate explicitly the 2-class $\hat{\mu}_\ell $ in $C_2(\pi_L, \partial \pi_L;\Z ). $
Since the arc $\beta_i $ has the same link-component to
$\gamma_{\ell_i}$ for some $\ell_i$,
we take the 2-cycle $ \nu_{\ell_i, \beta_i } $ obtained in Lemma \ref{2g2s22}.
Put a 2-cycle $d_i$ of the form
$$ \bigl( ( \mathfrak{m}_{\ell_i}, \mathfrak{m}_{\ell_i}^{-1}) , 1) -
( (1,1), \mathfrak{m}_{\ell_i}) - ((1,1), \mathfrak{m}_{\ell_i} ), $$
where the second (resp. third) term has only a non-trivial element in the $( i+1) $-th (resp. $\beta_i $-th) component.
Using the 2-cycle $\hat{\mu}_\ell^{\rm pre}$,
let us set $ \hat{\mu}_\ell := \hat{\mu}_\ell^{\rm pre} + \sum_{i=2}^{N_{\ell}} (\nu_{\ell_i, \beta_i } -d_i) $.
By construction, $ \hat{\mu}_\ell$ is also a 2-cycle and is presented by some elements of $(\pi_L,\partial \pi_L) $;
consequently, it lies in $ C_2(\pi_L,\partial \pi_L).$
Furthermore, from the definitions of $ \zeta_{\ell }$ and $\zeta$, we have the equality
$$ \langle \bigl( \zeta \circ \Omega(\mathcal{C}) \bigr) \smile \bigl( \zeta' \circ \Omega' (\mathcal{C}')\bigr) , \ \hat{\mu}_\ell\rangle =
\langle \bigl( \zeta_{\ell } \circ \Omega (\mathcal{C}) \bigr)\smile \bigl( \zeta_{\ell }' \circ \Omega' (\mathcal{C}') \bigr), \ \hat{\mu}^{\rm pre}_\ell \rangle
\in (M \otimes M')_{\pi_L }. $$
Notice from Lemma \ref{2g2s} that the pairing with $\psi$ is equal to $\mathcal{Q}_{\psi}(\mathcal{C}, \mathcal{C}') $.
Hence, the proof is completed. 
\end{proof}

\begin{proof}[Proof of Proposition \ref{aa1133c}]
Since $W$ is $S^2$ with removed $m$ open discs,
$W$ and $\partial W$ are Eilenberg-MacLane spaces,
and we have the isomorphisms $\pi_L = \pi_1( S^3 \setminus T_{m,m}) \cong \Z \times \pi_1(W ) \cong \Z \times F_{m-1}. $
Here the summand $\Z$ is generated by $a_1 \cdots a_m \in \pi_1( S^3 \setminus T_{m,m})$.
Hence, it follows from the assumption $z_1 \cdots z_m = \id_M$ and Lemma \ref{clAl1}
that the projection $ \mathcal{P}: \pi_1( S^3 \setminus T_{m,m}) \ra \pi_1(W ) $
induces an isomorphism $ \mathcal{P}^*: H^1( W,\partial W ; M ) \ra H^1(\pi_L ,\partial \pi_L ; M )$. 
Hence, the required claims immediately follow from Theorem \ref{mainthm2},
which completes the proof.
\end{proof}
\subsection*{Acknowledgments}
The author sincerely expresses his gratitude to
Akio Kawauchi for
many useful discussions on the classical Blanchfield pairing.
He also thanks 
Takahiro Kitayama and Masahico Saito 
for valuable comments.
The work is partially supported by JSPS KAKENHI Grant Number 00646903.

\vskip 1pc

\normalsize

Faculty of Mathematics, Kyushu University,
744, Motooka, Nishi-ku, Fukuoka, 819-0395, Japan

\

E-mail address: {\tt nosaka@math.kyushu-u.ac.jp}

\end{document}